\documentclass{amsart}
\usepackage{amsmath,amsfonts,amssymb}
\usepackage{latexsym}
\usepackage{epsfig,overpic}
\usepackage{stmaryrd}
\usepackage{color}
\usepackage{graphicx}

\numberwithin{equation}{section}

\def\Xint#1{\mathchoice
{\XXint\displaystyle\textstyle{#1}}%
{\XXint\textstyle\scriptstyle{#1}}%
{\XXint\scriptstyle\scriptscriptstyle{#1}}%
{\XXint\scriptscriptstyle\scriptscriptstyle{#1}}%
\!\int}
\def\XXint#1#2#3{{\setbox0=\hbox{$#1{#2#3}{\int}$ }
\vcenter{\hbox{$#2#3$ }}\kern-.6\wd0}}

\def\dashint{\Xint-}

\newcommand{\E}{\mathcal E}
\newcommand{\T}{\mathcal T}

\newcommand{\Div}{\mathop{\rm div}}
\newcommand{\curl}{\mathop{\rm curl}}

\definecolor{darkred}{rgb}{.7,0,0}

\definecolor{green}{rgb}{0,0.7,0}

\def\B{\mathfrak{B}}
\def\H{\mathfrak{H}}
\def\Z{\mathfrak{Z}}

\def\trace{\hspace{1pt}{\rm tr}\hspace{2pt}}
\def\osc{\hspace{1pt}{\rm osc}\hspace{1pt}}

\def\d{\hspace{2pt}{\rm d}}

\newcommand{\la}{\langle}
\newcommand{\ra}{\rangle}

\newtheorem{theorem}{Theorem}
\newtheorem{lemma}[theorem]{Lemma}

\newtheorem{remark}{Remark}

\title[A posteriori FEEC]{A posteriori error estimates for finite element exterior calculus: The de Rham complex}

\author[A~.Demlow]{Alan Demlow$^1$}

\thanks{$^1$Partially supported by National Science Foundation grant DMS-1016094.}

\address{
  Department of Mathematics,
  University of Kentucky,
  715 Patterson Office Tower,
  Lexington, KY 40506,
  USA}

\email{alan.demlow@uky.edu}

\author[A.N.~Hirani]{Anil N. Hirani$^2$}

\thanks{$^2$Partially supported by National Science Foundation grant DMS-0645604.}

\address{Department of Computer Science, University of Illinois at Urbana-Champaign, 201 N. Goodwin Avenue, Urbana, IL 61801, USA}

\email{hirani@cs.illinois.edu}

\keywords{Finite element methods, exterior calculus, a posteriori error estimates, adaptivity}

\subjclass{65N15, 65N30}

\begin{document}

\begin{abstract}  Finite element exterior calculus (FEEC) has been developed over the past decade as a framework for constructing and analyzing stable and accurate numerical methods for partial differential equations by employing differential complexes.   The recent work of Arnold, Falk and Winther \cite{ArFaWi2010} includes a well-developed theory of finite element methods for Hodge Laplace problems, including a priori error estimates.  In this work we focus on developing a posteriori error estimates in which the computational error is bounded by some  computable functional of the discrete solution and problem data.  More precisely, we prove a posteriori error estimates of residual type for Arnold-Falk-Winther mixed finite element methods for Hodge-de Rham Laplace problems.  While a number of previous works consider a posteriori error estimation for Maxwell's equations and mixed formulations of the scalar Laplacian, the approach we take is distinguished by unified treatment of the various Hodge Laplace problems arising in the de Rham complex, consistent use of the language and analytical framework of differential forms, and the development of a posteriori error estimates for harmonic forms and the effects of their approximation on the resulting numerical method for the Hodge Laplacian.
 \end{abstract}



\maketitle


\section{Introduction}

In this paper we study a posteriori error estimation for finite element methods for the Hodge Laplacian for the de Rham complex generated by the Finite Element Exterior Calculus (FEEC) framework of Arnold, Falk, and Winther (abbreviated AFW below).  Finite element exterior calculus has been developed over the past decade as a general framework for constructing and analyzing mixed finite element methods for approximately solving partial differential equations.  In mixed methods two or more variables are approximated simultaneously, for example, stresses and displacements in elasticity or pressures and velocities in fluid problems.  The essential feature of FEEC is that differential complexes are systematically used in order to develop and analyze stable and efficient numerical methods.  Historically speaking, some aspects of mixed finite element theory such as the so-called ``commuting diagram property'' (cf. \cite{BF91}) are related to differential complexes, and some early work by geometers such as Dodziuk~\cite{Dodziuk1976} and computational electromagnetics researchers such as Bossavit and
others~\cite{Bossavit1988a} also contains ideas related to finite element exterior calculus.  However, around 2000 researchers working especially in electromagnetics and elasticity~\cite{Hiptmair1999,
  Arnold2002} independently began to realize that differential complexes can be systematically exploited in the numerical analysis of PDEs.  This work has culminated in the recent publication of the seminal work of Arnold, Falk, and Winther~\cite{ArFaWi2010} containing a general framework for FEEC (cf. also \cite{ArFaWi2006}).

Error analysis of numerical methods for PDE is generally divided into two categories, a priori and a posteriori.  To fix thoughts, let $u$ solve $-\Delta u=f$ in a polygonal domain $\Omega$ with Neumann boundary conditions $\frac{\partial u}{\partial n}=0$ on $\partial \Omega$ and $\int_\Omega u=0$ assumed in order to guarantee uniqueness.  Also let $u_h \in S_h$ be a finite element approximation to $u$, where $S_h \subset H^1(\Omega)$ is the continuous piecewise polynomials of fixed degree $r$ with respect to a mesh $\T_h$.  A classical a priori estimate is
\begin{equation}\label{eq1-1}
\|u-u_h\|_{H^1} \le C h^r \|u\|_{r+1}.
\end{equation}
Such estimates are useful for verifying optimality of methods with respect to polynomial degree and are commonly used to verify code correctness.  However, they provide no information about the actual size of the computational error in any given practical problem and often assume unrealistic regularity of the unknown solution.  A posteriori error estimates provide a complementary error analysis in which the error is bounded by a computable functional of $u_h$ and $f$:
\begin{equation} \label{eq1-2}
\|u-u_h\| \le \E(u_h,f).
\end{equation}
Such estimates provide no immediate information about asymptotic error decrease, but do ideally yield concrete and reliable information about the actual size of the error in computations.  In addition, $\E(u_h, f)$ and related quantities are typically used to derive adaptive finite element methods in which information from a given computation is used to selectively refine mesh elements in order to yield a more efficient approximation.  We do not directly study adaptivity here.  

While there are many types of a posteriori error estimators \cite{AO00, BR03}, we focus our attention on residual-type error estimators.  Roughly speaking, residual estimators are designed to control $u-u_h$ by controlling the residual $f+ \Delta u_h$, which is not a function (since $\nabla u_h$ is only piecewise continuous) but is a functional lying in the dual space of $H^1(\Omega)\slash \mathbb{R}$.  Given a triangle $K \in \T_h$, let $h_K={\rm diam}(T)$.  We define the {\it elementwise a posteriori error indicator}%
\begin{equation}\label{eq1-3} \eta(K)=h_K \|f+\Delta u_h\|_{L_2(K)}+ h_K^{1/2} \|\llbracket \nabla u_h \rrbracket \|_{L_2(\partial K)}.  \end{equation}%
The {\it volumetric residual} $h_K \|f+\Delta u_h\|_{L_2(K)}$ may roughly be seen as bounding the regular portion of the residual $f+ \Delta u_h$.  
$\llbracket \nabla u_h \rrbracket$ is defined as the jump in the normal component of $\nabla u_h$ across interior element boundaries and as $\nabla u_h \cdot n$ on element faces $e \subset \partial \Omega$.  Since natural boundary conditions are satisfied only approximately in the finite element method, the latter quantity is not generally 0.  The corresponding term  in \eqref{eq1-3} may be thought of as measuring the singular portion of the distribution $f+ \Delta u_h$.  A standard result is that under appropriate assumptions on $\T_h$,
\begin{equation} \label{eq1-4}
\|u-u_h\|_{H^1(\Omega) \slash \mathbb{R}} \le C ( \sum_{K \in \T_h} \eta(K)^2)^{1/2}.
\end{equation}
That is, $\E(u_h,f)=C (\sum_{K \in \T_h} \eta(K)^2)^{1/2}$ is a {\it reliable} error estimator for the energy error $\|u-u_h\|_{H^1(\Omega) \slash \mathbb{R}}$.  An error estimator $\E$ is said to be efficient if $\E(u_h,f) \le \tilde{C} \|u-u_h\|$, perhaps up to higher-order terms.  
Given $K \in \T_h$, let $\omega_K$ be the ``patch'' of elements touching $K$.  We also define $\osc(K)= h_K \|f-P f \|_{L_2(K)}$, where $Pf$ is the $L_2(K)$ projection onto the polynomials having degree one less than the finite element space.  Then
\begin{equation}\label{eq1-5}
\eta(K)^2 \le C \left (\|u-u_h\|_{H^1(\omega_K)}^2 + \sum_{K \subset \omega_K} \osc(K)^2 \right).
\end{equation}
In our development below we recover \eqref{eq1-4} and \eqref{eq1-5} and also develop similar results for other Hodge Laplace problems such as the vector Laplacian.  

We pause to remark that residual estimators are usually relatively rough estimators in the sense that the ratio $\E(u_h,f)/\|u-u_h\|$ is often not close to 1 as would be ideal, and there are usually unknown constants in the upper bounds.   However, they have a structure closely related to the PDE being studied, generally provide unconditionally reliable error estimates up to constants, and can be used as building blocks in the construction and analysis of sharper error estimators.   Thus they are studied widely and often used in practice.  

In this work we prove a posteriori error estimates for mixed finite element methods for the Hodge Laplacian for the de Rham complex.  Let $H \Lambda^0 \overset{d}{\rightarrow} H \Lambda^1 \overset{d}{\rightarrow} \cdots \overset{d}{\rightarrow} H \Lambda^{n-1} \overset{d} \rightarrow L_2$ be the $n$-dimensional de Rham complex.  Here $\Lambda^k$ consists of $k$-forms and $H \Lambda^k$ consists of $L_2$-integrable $k$-forms $\omega$ with $L_2$ integrable exterior derivative $d \omega$.  For $n=3$, the de Rham complex is $H^1 \overset{\nabla}{\rightarrow} H({\rm curl}) \overset{\rm curl}{\rightarrow} H({\rm div}) \overset{{\rm div}}{\rightarrow} L_2$.  For $0 \le k \le n$, the Hodge Laplacian problem is given by $\delta d u+ d \delta u=f$, where $\delta$ is the adjoint (codifferential) of the exterior derivative $d$.  When $n=3$, the $0$-Hodge Laplacian is the standard scalar Laplacian, and the AFW mixed formulation reduces to the standard weak formulation of the Laplacian with natural Neumann boundary conditions.  The $1$- and $2$-Hodge Laplacians are instances of the vector Laplacian $\curl \curl-\nabla \Div$ with different boundary conditions, and the corresponding FEEC approximations are mixed approximations to these problems.  The $3$-Hodge Laplacian is again the scalar Laplacian, but the AFW mixed finite element method now coincides with a standard mixed finite element method such as the Raviart-Thomas formulation, and Dirichlet boundary conditions are natural.  We also consider essential boundary conditions below.

Next we briefly outline the scope of our results and compare them with previous work.  First, in the context of mixed methods for the scalar Laplacian and especially FEM for Maxwell's equations two technical tools have proved essential for establishing a posteriori error estimates.  These are regular decompositions \cite{Hip02, PZ02} and locally bounded commuting quasi-interpolants \cite{Sch08}.  Relying on recent analytical literature and modifying existing results to meet our needs, we provide versions of these tools for differential forms in arbitrary space dimension.  Next, our goal is to prove a posteriori estimates simultaneously for mixed approximations to all $k$-Hodge Laplacians ($0 \le k \le n$) in the de Rham complex.  Focusing individually on the various Hodge Laplace operators, we are unaware of previous proofs of a posteriori estimates for the  vector  Laplacian, although a posteriori estimates for Maxwell's equations are well-represented in the literature; \cite{BHHW00, Sch08, ZCSWX11} among many others.  The estimators that we develop for the standard mixed formulation for the well-studied case of the scalar Laplacian are also modestly different from those previously appearing in the literature (cf. \S\ref{subsec_classicalmixed} below).   In addition, our work extends beyond the two- and three-dimensional setting assumed in these previous works. 
Throughout the paper we also almost exclusively use the notation and language of, and analytical results for, differential forms.  The only exception is \S\ref{sec_examples}, where we use standard notation to write down our results for all four three-dimensional Hodge-Laplace operators.  This use of differential forms enables us to systematically highlight properties of finite element approximations to Hodge Laplace problems in a unified fashion.  
A final unique feature of our development is our treatment of harmonic forms.  In \S\ref{subsec_harmonic}, we give an abstract framework for bounding a posteriori the {\it gap} between the spaces $\H^k$ of continuous forms and $\H_h^k$ of discrete harmonic forms (defined below). This framework is an important part of our theory for the Hodge Laplacian and is also potentially of independent interest in situations where harmonic forms are a particular focus. Since our results include bounds for the error in approximating harmonic forms, our estimators also place no restrictions on domain topology.  We are not aware of previous works where either errors in approximating harmonic forms or the effects of such errors on the approximation of related PDE are analyzed a posteriori.   

We next briefly describe an interesting feature of our results.  The AFW mixed method for the $k$-Hodge Laplace problem simultaneously approximates the solution $u$, $\sigma=\delta u$, and the projection $p$ of $f$ onto the harmonic forms by a discrete triple $(\sigma_h, u_h, p_h)$.  The natural starting point for error analysis is to bound the $H \Lambda^{k-1} \times H \Lambda^k \times L_2$ norm 
of the error, since this is the variational norm naturally related to the ``inf-sup'' condition used to establish stability for the weak mixed formulation.  Abstract a priori bounds for this quantity are given in Theorem 3.9 of \cite{ArFaWi2010} (cf. \eqref{eq2-6} below), and we carry out a posteriori error analysis only in this natural mixed variational norm.  Aside from its natural connection with the mixed variational structure, this norm yields control of the error in approximating the Hodge decomposition of the data $f$ when $1 \le k \le n-1$, which may be advantageous.

    As in the a priori error analysis, the natural variational norm has some disadvantages.  Recall that residual estimators for the scalar Laplacian bound the residual $f + \Delta u_h$ in a negative-order Sobolev norm.  For approximations of the vector Laplacian, establishing efficient and reliable a posteriori estimators in the natural norm requires that different portions of the Hodge decomposition of the residual $f-d \sigma_h-p_h -\delta d u_h$ be measured in different norms.  Doing so requires access to the Hodge decomposition of $f$, but it is rather restrictive to assume access to this decomposition a priori.  
 We are able to access the Hodge decomposition of $f$ weakly in our estimators below, but at the expense of requiring more regularity of $f$ than is needed to write the Hodge Laplace problem (cf. \S\ref{sec4-2}).  In the a priori setting it is often possible to obtain improved error estimates by considering the discrete variables and measuring their error separately in weaker norms; cf. Theorem 3.11 of \cite{ArFaWi2010}.  This is an interesting direction for future research in the a posteriori setting as it may help to counteract this ``Hodge imbalance'' in the residual.


The paper is organized as follows.  In Section~\ref{sec_hilbert} we review the Hilbert complex structure employed in finite element exterior calculus, begin to develop a posteriori error estimates using this structure, and establish a framework for bounding errors in approximating harmonic forms.  In Section~\ref{sec_derham}, we recall details about the de Rham complex and also prove some important auxiliary results concerning commuting quasi-interpolants and regular decompositions.  Section \ref{sec_reliability} contains the main theoretical results of the paper, which establish a posteriori upper bounds for errors in approximations to the Hodge Laplacian for the de Rham complex.  Section \ref{sec_efficiency} contains corresponding elementwise efficiency results.  In Section~\ref{sec_examples} we demonstrate how our results apply to several specific examples from the three-dimensional de Rham complex and where appropriate compare our estimates to previous literature.  


\section{Hilbert complexes, harmonic forms, and abstract error
  analysis} \label{sec_hilbert}

In this section we recall basic definitions and properties of Hilbert
complexes, then begin to develop a framework for a posteriori error
estimation.

\subsection{Hilbert complexes and the abstract Hodge
  Laplacian} \label{sec2-1}

The definitions in this section closely follow \cite{ArFaWi2010},
which we refer the reader to for a more detailed presentation.  We
assume that there is a sequence of Hilbert spaces $W^k$ with inner
products $\langle \cdot, \cdot \rangle$ and associated norms $\| \cdot
\|$ and closed, densely defined linear maps $d^k$ from $W^k$ into
$W^{k+1}$ such that the range of $d^k$ lies in the domain of $d^{k+1}$
and $d^{k+1} \circ d^{k}=0$.  These form a Hilbert complex $(W,d)$.
Letting $V^k \subset W^k$ be the domain of $d^k$, there is also an
associated domain complex $(V,d)$ having inner product $\la u,v
\ra_{V^k}=\la u,v \ra_{W^k}+\la d^k u, d^k v \ra_{W^{k+1}}$ and
associated norm $\| \cdot \|_V$.  The complex $...\rightarrow V^{k-1}
\rightarrow V^k \rightarrow V^{k+1} \rightarrow ...$ is then bounded
in the sense that $d^k$ is a bounded linear operator from $V^k$ into
$V^{k+1}$.


The kernel of $d^k$ is denoted by $\Z^k=\B^k \oplus \H^k$, where
$\B^k$ is the range of $d^{k-1}$ and $\H^k$ is the space of harmonic
forms $\B^{k \perp_W} \cap \Z^k$. The Hodge decomposition is an
orthogonal decomposition of $W^k$ into the range $\B^k$,
harmonic forms $\H^k$, and their orthogonal complement $\Z^{k
  \perp_W}$.  Similarly, the Hodge decomposition of $V^k$ is
\begin{equation}\label{eq2-0}
V^k = \B^k \oplus \H^k \oplus \Z^{k \perp},
\end{equation}
where henceforth we simply write $\Z^{k \perp}$ instead of $\Z^{k
  \perp_V}$ except as noted.  The dual complex consists of the same
spaces $W^k$, but now with increasing indices, along with the
differentials consisting of adjoints $d_k^*$ of $d^{k-1}$.  The domain
of $d_k^*$ is denoted by $V_k^*$, which is dense in $W^k$.

The Poincar\'{e} inequality also plays a fundamental role; it reads
\begin{equation} \label{eq2-1}
\|v\|_V \lesssim \|d^k v\|_W, ~v \in \Z^{k \perp}.  
\end{equation}
Here and in what follows, we write $a \lesssim b$ when $a \le Cb$ with
a constant $C$ that does not depend on essential quantities.  Finally,
we assume throughout that the complex $(W,d)$ satisfies the
compactness property described in \S3.1 of \cite{ArFaWi2010}.

The immediate goal of the finite element exterior calculus framework
presented in \cite{ArFaWi2010} is to solve the ``abstract Hodge
Laplacian'' problem given by $Lu=(dd^*+d^* d)u=f$.  $L:W^k \rightarrow
W^k$ is called the Hodge Laplacian in the context of the de Rham
complex (in geometry, this operator is often called Hodge-de Rham
operator).  This problem is uniquely solvable up to harmonic forms when $f
\perp \H^k$.  It may be rewritten in a well-posed weak mixed
formulation as follows.  Given $f \in W^k$, we let $p=P_{\H^k} f$ be
the harmonic portion of $f$ and solve $Lu=f-p$.  In order to ensure
uniqueness, we require $u \perp \H^k$.  Writing $\sigma=d^*u$, we thus
seek $(\sigma, u, p) \in V^{k-1} \times V^k \times \H^k$ solving
\begin{equation}\label{eq2-2}
\begin{array}{rcll}
  \langle \sigma, \tau \rangle- \langle d \tau, u \rangle &=& 0, & 
  \tau \in V^{k-1},
  \\ \langle d \sigma, v \rangle + \langle d u, dv \rangle + 
  \langle v, p \rangle &=& \langle f, v \rangle,~ & v \in V^k,
  \\ \langle u, q \rangle &=& 0 , & q \in \H^k.
\end{array}
\end{equation}

So-called inf-sup conditions play an essential role in analysis of mixed formulations.  We define $B(\sigma, u, p; \tau, v,q)=\la \sigma, \tau \ra-\la d\tau, u \ra+ \la d \sigma, v  \ra + \la du, dv \ra + \la v, p \ra - \la u, q \ra$, which is a bounded bilinear form on $[V^{k-1} \times V^k \times \H^k] \times [V^{k-1} \times V^k \times \H^k]$.   We will employ the following, which is Theorem 3.1 of \cite{ArFaWi2010}.

\begin{theorem}  Let $(W,d)$ be a closed Hilbert complex with domain complex $(V,d)$.  There exists a constant $\gamma>0$, depending only on the constant in the Poincar\'{e} inequality \eqref{eq2-1}, such that for any $(\sigma, u, p) \in V^{k-1} \times V^k \times \H^k$ there exists $(\tau, v, q) \in V^{k-1} \times V^k \times \H^k$ such that
\begin{equation}\label{eq2-3}
B(\sigma, u, p ; \tau, v, q) \ge \gamma ( \|\sigma\|_V  + \|u\|_V + \|p\|)(\|\tau\|_V + \|v\|_V + \|q\|).
\end{equation}
\end{theorem}

\subsection{Approximation of solutions to the abstract Hodge Laplacian}

Assuming that $(W,d)$ is a Hilbert complex with domain complex $(V,d)$
as above, we now choose a finite dimensional subspace $V_h^k \subset
V^k$ for each $k$.  We assume also that $d V_h^k \subset V_h^{k+1}$,
so that $(V_h^k,d)$ is a Hilbert complex in its own right and a
subcomplex of $(V,d)$.  It is important to note that while the
restriction of $d$ to $V_h^k$ acts as the differential for the
subcomplex, $d^*$ and the adjoint $d_h^*$ of $d$ restricted to $V_h^k$
do not coincide. The discrete adjoint $d_h^*$ does not itself play a
substantial role in our analysis, but the fact that it does not
coincide with $d^*$ should be kept in mind.

The Hodge decomposition of $V_h^k$ is written 
\begin{equation}\label{eq2-4}
V_h^k=\B_h^k \oplus \H_h^k \oplus \Z_h^{k\perp}.
\end{equation}
Here $\B_h^k=d V_h^{k}$, with similar definitions of $\H_h^k$ and
$\Z_h^{k \perp}$ where $\perp$ is in $V_h$.  This discrete Hodge
decomposition plays a fundamental role in numerical methods, but it
only partially respects the continuous Hodge decomposition
\eqref{eq2-0}.  In particular, we have:
\begin{align}\label{eq2-4-a}
\begin{aligned}
\B_h^k &\subset \B^k,
\\ \H_h^k & \subset \Z^k \hbox{ but } \H_h^k \not\subset \H^k,
\\ \Z_h^{k \perp} & \not\subset \Z^{k \perp}. 
\end{aligned}
\end{align}

Bounded cochain projections play an essential role in finite element exterior calculus.  We assume the existence of an operator $\pi_h: V^k \rightarrow V_h^k$ which is bounded in both the $W$-norm $\|\cdot\|$ and the $V$-norm $\|\cdot\|_V$ and which commutes with the differential:  $d^{k} \circ \pi_h^k=\pi_h^{k+1} \circ d^k$.  In contrast to the a priori analysis of \cite{ArFaWi2010}, our a posteriori analysis does not require that $\pi_h$ be a projection, that is, we do not require that $\pi_h$ act as the identity on $V_h$.  In more concrete situations we shall however require certain other properties that are not needed in a priori error analysis.  

Approximations to solutions to \eqref{eq2-2} are constructed as follows.  Let $(\sigma_h, u_h, p_h) \in
V_h^{k-1} \times V_h^k \times \H_h^k$ satisfy
\begin{equation}\label{eq2-5}
  \begin{array}{rcll}
    \langle \sigma_h, \tau_h \rangle- \langle d \tau_h, u_h \rangle &=& 0,
    & \tau_h \in V_h^{k-1},\\ 
    \langle d \sigma_h, v_h \rangle + \langle d u_h, d v_h \rangle + 
    \langle v_h, p_h \rangle &=& \langle f, v_h \rangle,~ & v_h \in V_h^k,
    \\ \langle u_h, q_h \rangle &=& 0 , & q_h \in \H_h^k.
  \end{array}
\end{equation}
Existence and uniqueness of solutions to this problem are guaranteed by our assumptions.   A discrete inf-sup condition analogous to \eqref{eq2-3} with constant $\gamma_h$ depending on stability constants of the projection operator $\pi_h$ but otherwise independent of $V_h$ is contained in \cite{ArFaWi2010}; we do not state it as we do not need it for our analysis.  In addition, Theorem 3.9 of \cite{ArFaWi2010} contains abstract error bounds:   So long as the subcomplex $(V_h,d)$ admits uniformly $V$-bounded cochain projections,
\begin{align}\label{eq2-6}
\begin{aligned}
\|\sigma-\sigma_h\|_V & + \|u-u_h\|_V + \|p-p_h\| 
\\ & \lesssim \inf_{\tau \in V_h^{k-1}} \|\sigma-\tau\|_V + \inf_{v \in V_h^k} \|u-v\|_V 
\\ & ~~~~+ \inf_{q \in V_h^k} \|p-q\|_V + \tilde{\mu} \inf_{v \in V_h^k} \|P_\B u-v\|_V,
\end{aligned}
\end{align}
where $\tilde{\mu}=\sup_{r \in \H^k, \|r\|=1} \|(I-\pi_h^k) r
\|$. We will use the notation $P_S$ for the orthogonal projection onto
the subspace $S$ as in the case of $P_\B$ above.

\subsection{Abstract a posteriori error analysis} We next begin an a
posteriori error analysis, remaining for the time being within the
framework of Hilbert complexes.  A working principle of a posteriori
error analysis is that if a corresponding a priori error analysis
employs a given tool, one looks for an a posteriori ``dual'' of that
tool in order to prove corresponding a posteriori results.  Thus while
the proof of \eqref{eq2-6} employs a {\it discrete} inf-sup condition,
we shall employ the {\it continuous} inf-sup condition \eqref{eq2-3}.
Writing $e_\sigma=\sigma-\sigma_h$, $e_u=u-u_h$, and $e_p=p-p_h$, we
use the triangle inequality and \eqref{eq2-3} to compute
\begin{align}  \label{eq2-7}
  \begin{aligned}
    \|e_\sigma\|_V+& \|e_u\|_V+\|e_p\|   \le   
    \left ( \|e_\sigma\|_V+\|e_u\|_V+\|p-P_\H p_h \| \right ) +\|P_\H
    p_h-p_h\|
    \\   & \le  \frac{1}{\gamma} \sup_{\substack{(\tau, v, q) \in V^{k-1}
        \times V^k \times \H^k, \\ \|\tau\|_V+\|v\|_V+\|q\|=1}}
    B(e_\sigma, e_u, p-P_\H p_h; \tau, v, q) + \|P_\H p_h-p_h\|
    \\ &  \le  \frac{1}{\gamma} \sup_{\substack{(\tau, v, q) \in V^{k-1}
        \times V^k \times \H^k, \\ 
        \|\tau\|_V+\|v\|_V+\|q\|=1}}   \Big ( \langle e_\sigma, \tau
    \rangle - \langle d \tau, e_u \rangle + \langle d e_\sigma, v
    \rangle + \langle d e_u, dv \rangle 
    \\ & ~~~~ + \langle v, e_p \rangle + \langle e_u, q \rangle \Big )   +
    (1+\frac{1}{\gamma}) \|P_\H p_h-p_h\|.  
    \end{aligned}
    \end{align}
Employing Galerkin orthogonality implied by subtracting the first two lines of \eqref{eq2-5} and \eqref{eq2-2} in order to insert $\pi_h \tau$ and $\pi_h v$ into \eqref{eq2-7} and then again employing \eqref{eq2-2} finally yields
    \begin{align}
    \begin{aligned}
     \|e_\sigma\|_V+& \|e_u\|_V+\|e_p\| 
     \\ &  \le \frac{1}{\gamma} \sup_{\substack{(\tau, v, q) \in V^{k-1}
        \times V^k \times \H^k, \\ 
        \|\tau\|_V+\|v\|_V+\|q\|=1}}   \Big ( \langle \sigma_h, \tau-\pi_h \tau
    \rangle - \langle d ( \tau-\pi_h \tau), u_h \rangle 
    \\ & ~~~~+ \langle f-d\sigma_h -p_h, v-\pi_h v  \rangle - \langle d u_h, d(v-\pi_h v) \rangle 
     + \langle e_u, q \rangle \Big )  \\ & ~~~~ +
    (1+\frac{1}{\gamma}) \|P_\H p_h-p_h\|.
  \end{aligned}
\end{align}

The terms $ \langle \sigma_h, \tau-\pi_h \tau    \rangle - \langle d ( \tau-\pi_h \tau), u_h \rangle$ and $ \langle f-d\sigma_h-p_h, v-\pi_h v  \rangle +  \langle d u_h, d(v-\pi_h v) \rangle$ in \eqref{eq2-7} can be attacked in concrete situations with adaptations of standard techniques for residual-type a posteriori error analysis, but no further progress can be made on this abstract level without further assumptions on the finite element spaces.  The terms $\langle e_u, q \rangle$ and $(1+\frac{1}{\gamma}) \|P_\H p_h-p_h\|$, on the other hand, are nonzero only when $\H_h^k \neq \H^k$.  In this case \eqref{eq2-5} is a generalized Galerkin method, and further abstract analysis is helpful in elucidating how these nonconformity errors may be bounded.  We carry out this analysis in the following subsection.

\subsection{Bounding the ``harmonic errors''} \label{subsec_harmonic}
We next lay groundwork for bounding the terms $\|p_h-P_\H p_h\|$ and $\sup_{q \in \H^k} \la e_u, q \ra$.  Since $p_h \in \H_h^k$, \eqref{eq2-4-a} and $\Z^k=\B^k \oplus \H^k$ imply that $p_h-P_\H p_h \in \B^k$.  Recalling that $v \in \B^k$ implies that $v=d \phi$ for some $\phi \in V^{k-1}$ and also that $P_\H p_h \in \H^k \perp \B^k$ yields
\begin{align}\label{eq2-8}
\|p_h-P_\H p_h\|=\sup_{v \in \B^k, \|v\|=1} \la p_h -P_\H p_h , v \ra =\sup_{\phi \in V^{k-1}, \|d\phi\|=1} \la p_h, d \phi \ra.  
\end{align}
The discrete Hodge decomposition \eqref{eq2-4} implies that $\pi_h^k d
\phi=d \pi_h^{k-1} \phi \in \B_h^k \perp \H_h^k \ni p_h$.  Also note
that by the Poincar\'e inequality \eqref{eq2-1}, $\sup_{\phi \in V^{k-1}, \|d \phi\|=1} \la p_h, d \phi \ra$ is uniformly equivalent to $\sup_{\phi \in V^{k-1}, \|\phi\|_V=1} \la p_h, d \phi \ra$.  Thus 
\begin{equation} \label{eq2-9}
\|p_h-P_\H p_h\|  \lesssim  \sup_{\phi \in V^{k-1}, \|\phi\|_V=1} \la p_h, d (\phi-\pi_h \phi) \ra.
\end{equation}
We do not manipulate \eqref{eq2-9} any further without making more
precise assumptions about the spaces and exterior derivative involved.  Recall that the goal of \eqref{eq2-9} is to measure the amount by which the discrete
harmonic function $p_h$ fails to be a {\it continuous} harmonic
function.  If $p_h$ were in fact in $\H^k$, we would have $d^* p_h=0$, which would immediately imply that the right-hand-side of \eqref{eq2-9} is 0.  In \eqref{eq2-9} we measure the degree by which this is {\it not} true by testing weakly with a test function $d \phi$, minus a discrete approximation to the test function.

Before bounding the term $\sup_{q \in \H^h, \|q\|=1} \la e_u, q \ra$ we consider the {\it gap} between $\H^k$ and $\H_h^k$.  Given closed subspaces $A,B$ of a Hilbert space $W$, let
\begin{align}\label{eq2-10}
  \delta(A,B)=\sup_{x \in A, \|x\|=1} {\rm dist}(x, B)=\sup_{x \in A,
    \|x\|=1} \|x-P_{B}x\|.
  \end{align}
The gap between the subspaces $A$ and $B$ is defined as 
\begin{align}  \label{eq2-11}
  {\rm gap}(A,B)=\max(\delta(A,B), \delta(B, A)).
\end{align}

In the situation below, we will require information about
$\delta(\H^k, \H_h^k)$, but are able to directly derive
a posteriori bounds only for $\delta(\H_h^k, \H^k)$.  Thus it is
necessary to understand the relationship between $\delta(A,B)$ and
$\delta(B,A)$.  
   
\begin{lemma} \label{lem_gap}
  Assume that $A$ and $B$ are subspaces of the Hilbert space $W$, both
  having dimension $n<\infty$. Then
  \begin{equation}    \label{eq2-12}
    \delta(A,B)=\delta(B,A)={\rm gap}(A,B).
  \end{equation}
\end{lemma}
 
\begin{proof}The result is essentially found in \cite{Ka76}, Theorem 6.34, pp. 56-57. Assume first that $\delta(A,B)<1$.  The assumption that $\dim A=\dim B$ then implies that the nullspace of $P_B$ is 0 and that $P_B$ maps $A$ onto $B$ bijectively.  Thus Case i of Theorem 6.34 of \cite{Ka76} holds, and \eqref{eq2-12} follows from (6.51) of that theorem by noting that $\delta(A,B)=\|I-P_B\|_{(A,W)}=\|(I-P_B)P_A\|_{(W,W)}$.
 
If $\delta(A,B)=1$, then there is $0 \neq b \in B$ which is orthogonal
to $P_B(A)$.  Letting $\{a_i\}_{i=1,..,M}$ be an orthonormal basis for
$A$, we have $P_A b=\sum_{i=1}^M (a_i, b) a_i =\sum_{i=1}^M (P_B a_i,b
)a_i=0$, since $b \perp P_B(A)$.  Thus
$1=\|I-P_A\|_{(B,A)}=\delta(B,A)$, so that \eqref{eq2-12} holds in this
case also.  
\end{proof}

Thus we can bound ${\rm gap}(\H^k, \H_h^k)$
by bounding only $\delta(\H_h^k, \H^k)$, which we now turn our attention to.   First write $\delta(\H_h^k, \H^k)= \sup_{q_h \in \H_h^k, \|q_h\|=1} \|q_h-P_{\H^k} q_h\|$.  For a given $q_h \in \H_h^k$, we may employ exactly the same arguments as in \eqref{eq2-8} and \eqref{eq2-9} above to find
\begin{align} \label{eq2-13}
\|q_h-P_\H q_h\| \lesssim \sup_{ \phi \in V^{k-1}, \|\phi\|_V=1} \la q_h, d(\phi -\pi_h \phi) \ra.
\end{align}
We now let $\{q_1,...,q_M\}$ be an orthonormal basis for $\H_h^k$ and
assume that we have a posteriori bounds
\begin{align}\label{eq2-14}
\sup_{ \phi \in V^{k-1}, \|\phi\|_V=1} \la q_i, d(\phi -\pi_h \phi) \ra \le \mu_i, ~i=1,...,M.
\end{align}
We obtain such bounds for the de Rham complex below.  Given an arbitrary unit vector $q_h \in \H_h^k$, we write $q_h =\sum_{i=1}^M a_i q_i$, where $|\vec{a}|=1$.  Inserting this relationship
into \eqref{eq2-13} yields $\delta(\H_h^k, \H^k) \le \sup_{|\vec{a}|=1}  \sum_{i=1}^M a_i \mu_i$.  This expression is maximized by choosing $\vec{a}=\vec{\mu}/|\vec{\mu}|$, where $\vec{\mu}=\{\mu_1,.., \mu_M\}$.  Thus
\begin{align} \label{eq2-15}
\delta(\H_h^k, \H^k) \le |\vec{\mu}|.
\end{align}
%
Combining \eqref{eq2-15} with \eqref{eq2-12}, we thus also have
\begin{align}\label{eq2-16}
{\rm gap}(\H^k, \H_h^k) \le |\vec{\mu}|.
\end{align}

Now we turn our attention to bounding $\|P_\H u_h\|= \sup_{q \in \H^k, \|q\|=1} \la e_u, q \ra$.  Our analysis of this term is slightly unusual in that we suggest two possible approaches.  One is likely to be sufficient for most applications and is less computationally intensive.  The other more accurately reflects the actual size of the term at hand, but requires additional computational expense with possibly little practical payoff.  

We first describe the cruder approach.   Because $u_h \perp \H_h^k$, 
\begin{align}\label{eq2-17}
\begin{aligned}
\|P_\H u_h \|& =\sup_{q \in \H^k, \|q\|=1} \la q, u_h \ra = \sup_{q \in \H^k, \|q\|=1} \la q-P_{\H_h^k} q, u_h \ra 
 \\ & \le \delta( \H^k, \H_h^k) \|u_h\| = {\rm gap}(\H^k, \H_h^k) \|u_h\|.
\end{aligned}
\end{align}
\eqref{eq2-16} may then be used in order to bound ${\rm gap} (\H^k, \H_h^k)$.  

Next we describe the sharper approach.  Since $u_h \perp \H_h^k$, we have $u_h= \tilde{u}_h + u_h^\perp$,
where $\tilde{u}_h \in \B_h^k$ and $u_h^\perp \in \Z_h^{k\perp}$.
Since $\B_h^k \subset \B^k \perp \H^k$, $u_h^\perp \perp \H_h^k$, and
$\H_h^k$ and $\H^k$ are both perpendicular to $\Z^{k \perp}$, we thus have
for any $q \in \H^k$ with $\|q\|=1$ that
\begin{align}\label{eq2-18}
\begin{aligned}
\langle u_h, q \rangle & = \langle u_h^\perp, q \rangle
 = \langle u_h^\perp, q-P_{\H_h} q \rangle  = \langle u_h^\perp -P_{\Z^{\perp}} u_h^{\perp}, q-P_{\H_h} q \rangle 
\\ & \le \|u_h^\perp-P_{\Z^{k\perp}} u_h^\perp\| \|  q-P_{\H_h} q\| \le {\rm gap} (\H^k, \H_h^k) \|u_h^\perp-P_{\Z^{k\perp}} u_h^\perp\| .
\end{aligned}
\end{align} 
But
\begin{align} \label{eq2-19}
u_h^\perp-P_{\Z^{\perp}} u_h^\perp= P_{\B} u_h^\perp + P_{\H} u_h^\perp=P_\B u_h^\perp +
P_{\H} u_h.  
\end{align}
Here the relationship $P_{\H} u_h^\perp=P_{\H} u_h$ holds because $\B_h^k \subset \B^k$ and so $P_\H \tilde{u}_h=0$.  Thus $\|u_h^\perp-P_{\Z^\perp} u_h^\perp\| \le \|P_{\B} u_h^\perp\|+ \|P_\H u_h\|$.  $\|P_\B u_h^\perp\|$ may be bounded as in \eqref{eq2-9} and \eqref{eq2-13} above:
\begin{align}\label{eq2-20}
\begin{aligned}
\|P_\B u_h^\perp\| & = \sup_{\phi \in V^{k-1}, \|d \phi\|=1} \la u_h^\perp, d(\phi - \pi_h \phi) \ra 
\\ & \lesssim \sup_{\phi \in V^{k-1}, \|\phi\|_V=1} \la u_h^\perp, d(\phi -\pi_h \phi) \ra.
\end{aligned}
\end{align}

Assuming a posteriori bounds ${\rm gap}(\H^k, \H_h^k) \lesssim \mu$ and 
$\|P_{\B} u_h^\perp\| \lesssim \epsilon$, we thus have
\begin{align}\label{eq2-21}
\|P_{\H} u_h\| \lesssim \mu (\epsilon + \|P_\H u_h\|).
\end{align}
Inserting \eqref{eq2-17} into \eqref{eq2-21} then finally yields
\begin{align}\label{eq2-22}
\|P_{\H} u_h \| \lesssim \epsilon \mu + \mu^2 \|u_h\|.
\end{align}

We now discuss the relative advantages of \eqref{eq2-17} and \eqref{eq2-22}.  The corresponding term in the a priori bound \eqref{eq2-6} is $\tilde{\mu} \inf_{v \in V_h^k} \|P_\B u-v\|_V$, which is a bound for $\|P_{\H_h} u\|$ (note the symmetry between the a priori and a posteriori bounds).  The term $\tilde{\mu}$ (defined following \eqref{eq2-6}) is easily seen to be bounded by ${\rm gap}(\H^k, \H_h^k)$ at least in the case that $\pi_h$ is a $W$-bounded cochain projection.  Also, it is easily seen that $\tilde{\mu}$ is generally of the same or higher order than other terms in \eqref{eq2-6} when standard polynomial approximation spaces are used.  Carrying this over to the a posteriori context, \eqref{eq2-17} will yield a bound for $\|P_\H u_h\|$ that while crude is not likely to dominate the estimator or drive adaptivity in generic situations.  

If a sharper bound for $\|P_\H u_h\|$ proves desirable (e.g., if ${\rm gap}(\H^k, \H_h^k) \|u_h\|$ dominates the overall error estimator), then one can instead employ \eqref{eq2-22}.  This corresponds in the a priori setting to employing the term $\inf_{v \in V_h^k} \|P_\B u-v\|_V$ and is likely to lead to an asymptotically much smaller estimate for $\|P_\H u_h\|$.  However, computing the term $\epsilon$ in \eqref{eq2-22} requires computation of the discrete Hodge decomposition of $u_h$, which may add significant computational expense.  

\subsection{Summary of abstract bounds}

We summarize our results above in the following lemma.  

\begin{lemma}  Assume that $(W,d)$ is a Hilbert complex with subcomplex $(V_h, d)$ and commuting, $V$-bounded cochain operator $\pi_h:V \rightarrow V_h$, and in addition that $(\sigma, u, p)$ and $(\sigma_h, u_h, p_h)$ solve \eqref{eq2-2} and \eqref{eq2-5}, respectively.  Then for some $(\tau, v, q) \in V^{k-1} \times V^k \times \H^k$ with $\|\tau \|_V + \|v \|_V + \|q\|=1$ and some $\phi \in V^{k-1}$ with $\|\phi\|_V=1$, 
\begin{align}\label{eq2-23}
\begin{aligned}
\|e_\sigma\|_V& + \|e_u\|_V + \|e_p\|  \lesssim |\la e_\sigma, \tau-\pi_h \tau \ra - \la d(\tau-\pi_h \tau), e_u \ra| \\ & ~~+ | \la f-d \sigma_h -p_h, v-\pi_h v \ra - \la d u_h, d (v-\pi_h v) \ra | 
\\ & ~~~~+ | \la p_h, d (\phi-\pi_h \phi) \ra|  + \mu \|u_h^\perp -P_{\Z^\perp} u_h^\perp\|.
\end{aligned}
\end{align}
Here $\mu = (\sum_{i=1}^M \mu_i^2)^{1/2}$, where $\sup_{\phi \in V^{k-1}, \|\phi\|_V=1}  \la q_i, d(\phi - \pi_h \phi) \ra \lesssim \mu_i$ for an orthonormal basis $\{q_1, ...., q_M\}$ of $\H_h^k$.   For the last term in \eqref{eq2-23} we may either use the simple bound $\|u_h^\perp -P_{\Z^\perp} u_h^\perp\| \le \|u_h\|$ or employ the bound $\mu \|u_h^\perp -P_{\Z^\perp} u_h^\perp\| \lesssim \mu \epsilon + \mu^2 \|u_h\|$,  where 
\begin{align}\label{eq2-24}
\sup_{\phi \in V^{k-1}, \|\phi\|_V=1} \la u_h^\perp, d(\phi - \pi_h \phi) \ra \lesssim \epsilon.
\end{align}
\end{lemma}


\section{The de Rham complex and commuting quasi-interpolants} \label{sec_derham}

As above, we for the most part follow \cite{ArFaWi2010} in our notation.  Also as above, we shall often be brief in our description of concepts contained in \cite{ArFaWi2010} and refer the reader to \S4 and \S6 of that work for more detail.  

\subsection{The de Rham complex} \label{subsec3-1}
Let $\Omega$ be a bounded Lipschitz polyhedral domain in $\mathbb{R}^n$, $n \ge 2$.   Let $\Lambda^k(\Omega)$ represent the space of smooth $k$-forms on $\Omega$.  $\Lambda^k(\Omega)$ is endowed with a natural $L_2$ inner product $\la \cdot, \cdot \ra$ and $L_2$ norm $\| \cdot \|$ with corresponding space $L_2 \Lambda^k(\Omega)$.  Letting also $d$ be the exterior derivative, $H \Lambda^k(\Omega)$ is then the domain of $d^k$ consisting of $L_2$ forms $\Omega$ for which $d \omega \in L_2 \Lambda^{k+1}(\Omega)$; we denote by $\|\cdot \|_H$ the associated graph norm.  $(L_2 \Lambda^k(\Omega), d)$ forms a Hilbert complex (corresponding to $(W,d)$ in the abstract framework of the preceding section) with domain complex 
\begin{align}\label{eq3-1}
0 \, {\rightarrow}\,  H \Lambda^0(\Omega) \overset{d}{\rightarrow} H \Lambda^1(\Omega) \overset{d}{\rightarrow} \cdots \overset{d}{\rightarrow} H \Lambda^n(\Omega) \rightarrow 0
\end{align} 
corresponding to $(V,d)$ above.   In addition, we denote by $W_p^r \Lambda^k(\Omega)$ the corresponding Sobolev spaces of forms and set $H^r \Lambda^k(\Omega) = W_2^r \Lambda^k (\Omega)$.  Finally, for $\omega \subset \mathbb{R}^n$, we let $\|\cdot\|_\omega=\|\cdot \|_{L_2 \Lambda^k(\omega)}$ and $\|\cdot \|_{H, \omega}=\|\cdot \|_{H \Lambda^k(\omega)}$; in both cases we omit $\omega$ when $\omega=\Omega$.    

Given a mapping $\phi:\Omega_1 \rightarrow \Omega_2$, we denote by $\phi^*\omega \in \Lambda^k(\Omega_1)$ the pullback of $\omega \in \Lambda^k(\Omega_2)$, i.e., 
\begin{align}
\label{eq3-1-aa}
(\phi^* \omega)_x (v_1, ..., v_k)=\omega_{\phi(x)} (D \phi_x (v_1),..., D \phi_x(v_k)). 
\end{align}
The trace $\trace$ is the pullback of $\omega$ from $\Lambda^k(\Omega)$ to $\Lambda^k (\partial \Omega)$ under the inclusion.  $\trace$ is bounded as an operator $H \Lambda^k(\Omega) \rightarrow H^{-1/2} \Lambda^k (\partial \Omega)$ and $H^1 \Lambda^k (\Omega) \rightarrow H^{1/2} \Lambda^k (\partial \Omega)$, and thus also $H^1\Lambda^k(\Omega) \rightarrow L_2 \Lambda^k(\partial \Omega)$.  We may now define $\mathring{H} \Lambda^k(\Omega)=\{\omega \in H\Lambda^k(\omega):\trace \omega=0 \hbox{ on } \partial \Omega\}$.   In addition, we define the space $H_0^1\Lambda^k(\Omega)$ as the closure of $C_0^\infty \Lambda^k(\Omega)$ in $H^1 \Lambda^k(\Omega)$.  $H_0^1\Lambda^k(\Omega)$ essentially consists of forms which are 0 in every component on $\partial \Omega$, which is in general a stricter condition than $\trace \omega=0$.  

The wedge product is denoted by $\wedge$.  The Hodge star operator is denoted by $\star$ and for $\omega \in \Lambda^k$, $\mu \in \Lambda^{n-k}$ satisfies
\begin{align} \label{eq3-1-a} 
\omega \wedge \mu = \la \star \omega, \mu \ra {\rm vol}, ~~ \int_{\Omega_0} \omega \wedge \mu = \la \star \omega, \mu \ra_{L_2\Lambda^{n-k}(\Omega_0)}.
\end{align}
$\star$ is thus an isometry between $L_2\Lambda^k$ and $L_2 \Lambda^{n-k}$.  The coderivative operator $\delta:\Lambda^k \rightarrow \Lambda^{k-1}$ is defined by 
\begin{align} \label{eq3-2}
\star \delta \omega = (-1)^k d \star \omega.
\end{align}
Applying Stokes' theorem leads to the integration-by-parts formula
\begin{align}\label{eq3-3}
\la d \omega, \mu \ra = \la \omega, \delta \mu \ra + \int_{\partial \Omega} \trace \omega \wedge \trace \star \mu, ~\omega \in H \Lambda^{k-1},~ \mu \in H^1 \Lambda^k.  
\end{align}
The coderivative coincides with the abstract codifferential introduced in \S \ref{sec2-1} when $\trace_{\partial \Omega} \star \mu=0$.  That is, the domain of the adjoint $d^*$ of $d$ is the space $\mathring{H}^* \Lambda^k(\Omega)$ consisting of forms $\mu \in L_2\Lambda^k$ whose weak coderivative is in $L_2 \Lambda^{k-1}$ and for which $\trace \star \mu =0$.  We will also use the space $H^* \Lambda^k=\star(H \Lambda^{n-k})$ consisting of $L_2$ forms whose weak codifferential lies in $L_2$; note that $v \in H^* \Lambda^k$ implies that $\trace \star v \in H^{-1/2}$.  

The Hodge decomposition $L_2 \Lambda^k(\Omega)=\B^k \oplus \H^k \oplus \B_k^*$ consists of the range $\B^k=\{d \varphi : \varphi \in H \Lambda^{k-1}(\Omega)\}$, harmonic forms $\H^k=\{\omega \in H \Lambda^k(\Omega): d \omega=0, ~\delta \omega=0,~\trace \star \omega=0\}$, and range $\B_k^*=\{\delta \omega: \omega \in \mathring{H}^* \Lambda^{k+1}(\Omega)\}$ of $\delta$.  $\dim \H^k$ is the $k$-th Betti number of $\Omega$.  The mixed Hodge Laplacian problem corresponding to \eqref{eq2-2} now reads:  Find $(\sigma, u, p) \in H \Lambda^{k-1} \times H \Lambda^k \times \H^k$ satisfying
\begin{eqnarray}
\sigma & =&  \delta u, ~ d \sigma + \delta du=f-p \hbox{ in } \Omega,  \label{eq3-4}
\\ \trace \star u &=& 0, ~ \trace \star d u = 0 \hbox{ on } \partial \Omega,  \label{eq3-5}
\\ u &\perp& \H^k. \label{eq3-6}
\end{eqnarray}
The boundary conditions \eqref{eq3-5} are enforced naturally in the weak formulation \eqref{eq2-2} and so do not need to be built into the function spaces for the variational form.  The additional boundary conditions
\begin{align} \label{eq3-7}
\trace \star \sigma=0, ~\trace \star \delta d u=0 \hbox{ on } \partial \Omega
\end{align}
are also satisfied.  To see this, note that $d$ and $\trace$ commute since $\trace$ is a pullback, and that $\trace \star \sigma$ and $\trace \star \delta du$ are both well defined in $H^{-1/2}$ since $\delta \sigma=0$ and $\delta \delta d u=0$ imply that $\sigma, \delta du \in H^*$.  Thus by \eqref{eq3-2}, $\trace \star \sigma= \trace (-1)^k d \star u=(-1)^k d \hspace{2pt} \trace \star u=0$.  Similarly, $\trace \star \delta d u=\trace (-1)^k d \star du= (-1)^k d \trace \star d u=0$.  These relationships are roughly akin to noting that for a scalar function $u$, the boundary condition $u=0$ on $\partial \Omega$ implies that the tangential derivatives of $u$ along $\partial \Omega$ are also 0.  

We also consider the Hodge Laplacian with essential boundary conditions, that is:  Find $(\sigma, u, p) \in \mathring{H} \Lambda^{k-1} \times \mathring{H} \Lambda^k \times \mathring{\H}^k$ satisfying
\begin{eqnarray}
\sigma & =&  \delta u, ~ d \sigma + \delta du=f-p \hbox{ in } \Omega,  \label{eq3-7a}
\\ \trace \sigma &=& 0, ~ \trace u = 0 \hbox{ on } \partial \Omega,  \label{eq3-7b}
\\ u &\perp& \mathring{\H}^k. \label{eq3-7c}
\end{eqnarray}
This is the Hodge Laplace problem for the de Rham sequence \eqref{eq3-1} with each instance of $H\Lambda^k(\Omega)$ replaced by $\mathring{H} \Lambda^k(\Omega)$.  Here we have denoted the corresponding parts of the Hodge decomposition using similar notation, e.g., $\mathring{\H}^k = \{ \omega \in \mathring{\Z}^k | \langle \omega, \mu \rangle = 0, ~ \mu \in \mathring{\B}^k \}$.  


\subsection{Finite element approximation of the de Rham complex}  Let $\T_h$ be a shape-regular simplicial decomposition of $\Omega$.  That is, for any $K_1, K_2 \in \T_h$, $\overline{K}_1 \cap \overline{K}_2$ is either empty or a complete subsimplex (edge, face, vertex, etc.) of both $K_1$ and $K_2$, and in addition all $K \in \T_h$ contain and are contained in spheres uniformly equivalent to $h_K :={\rm diam}(K)$.  

Denote by $(V_h, d)$ any of the complexes of finite element differential forms consisting of $\mathcal{P}_r$ and $\mathcal{P}_r^-$ spaces described in \S5 of \cite{ArFaWi2010}.   We do not give a more precise definition as we only use properties of these spaces which are shared by all of them.  The finite element approximation to the mixed solution $(\sigma, u, p)$ of the Hodge Laplacian problem is denoted by $(\sigma_h, u_h, p_h) \in V_h^k \times V_h^{k-1} \times \H_h^k$ and is taken to solve \eqref{eq2-5}, but now within the context of finite element approximation of the de Rham complex.  In order to solve \eqref{eq3-7a}--\eqref{eq3-7c} we naturally employ spaces $\mathring{V}_h^k = V_h^k \cap \mathring{H} \Lambda^k(\Omega)$ of finite element differential forms.  

\subsection{Regular decompositions and commuting quasi-interpolants} \label{subsec_interps}


We also employ a regular decomposition of the form $\omega= d \varphi + z$, where $\omega \in H \Lambda$ only, but $\varphi, z \in H^1$.  In the context of Maxwell's equations the term ``regular decomposition'' first appeared in the numerical analysis literature in the survey  \cite{Hip02} by Hiptmair (although similar results were previously available in the analysis literature).  Published at about the same time, the paper \cite{PZ02} of Pasciak and Zhao contains a similar result for $H(\curl)$ spaces which is also often cited in this context.  Below we rely on the paper \cite{MiMiMo08} of Mitrea, Mitrea, and Monniaux, which contains regularity results for certain boundary value problems for differential forms that may easily be translated into regular decomposition statements.   Recent work of Costabel and McIntosh \cite{CoMc10} contains similar results for forms, though with handling of boundary conditions that seems slightly less convenient for our purposes.  

We first state a lemma concerning the bounded invertibility of $d$; this is a special case of Theorem 1.5 of \cite{MiMiMo08}.  

\begin{lemma} \label{lem3-0}
Assume that $B$ is a bounded Lipschitz domain in $\mathbb{R}^n$ that is homeomorphic to a ball.  Then the boundary value problem $d\varphi=g \in L_2 \Lambda^k(B)$ in $B$, $\trace \varphi=0$ on $\partial B$ has a solution $\varphi \in H_0^1 \Lambda^{k-1} (B)$ with $\|\varphi\|_{H^1 \Lambda^{k-1} (B)} \lesssim \|g\|_B$ if and only if $dg=0$ in $B$, and in addition,  $\trace g=0$ on $\partial B$ if $0 \le k \le n-1$ and $\int_B g=0$ if $k=n$.  
\end{lemma}

Employing Lemma \ref{lem3-0}, we obtain the following regular decomposition result.  

\begin{lemma} \label{lem3-0-b}
Assume that $\Omega$ is a bounded Lipschitz domain in $\mathbb{R}^n$, and let $0 \le k \le n-1$.  Given $v \in H \Lambda^k(\Omega)$, there exist $\varphi \in H^1 \Lambda^{k-1} (\Omega)$ and $z \in H^1 \Lambda^k(\Omega)$ such that $v = d \varphi + z$, and 
\begin{align}\label{eq3-8-a}
\|\varphi \|_{H^1 \Lambda^{k-1}(\Omega)} + \|z\|_{H^1 \Lambda^k(\Omega)} \lesssim \|v\|_{H \Lambda^k(\Omega)}.
\end{align}
Similarly, if $v \in \mathring{H} \Lambda^k(\Omega)$, then there exist $\varphi \in H_0^1 \Lambda^{k-1}(\Omega)$ and $z \in H_0^1 \Lambda^k(\Omega)$ such that $v = d \varphi + z$ and \eqref{eq3-8-a} holds.  In the case $k=n$, $H \Lambda^n(\Omega)$ is identified with $L_2 \Lambda^n(\Omega)$.   The same results as above hold with the exception that $z \in L_2 \Lambda^n (\Omega)$ only and satisfies $\|z\|_{L_2 \Lambda^n(\Omega)} \lesssim \|v\|_{H \Lambda^k(\Omega)}$.  
\end{lemma}

\begin{proof}

We first consider the case $v \in H \Lambda^k(\Omega)$.  By Theorem A of \cite{MiMiSh08}, the assumption that $\partial \Omega$ is Lipschitz implies the existence of a bounded extension operator $E: H \Lambda^k (\Omega) \rightarrow H \Lambda^k (\mathbb{R}^n)$.  Without loss of generality, we may take $E \omega$ to have compact support in a ball $B$ compactly containing $\Omega$, since if not we may multiply $E v$ by a fixed smooth cutoff function that is 1 on $\Omega$ and still thus obtain an $H\Lambda$-bounded extension operator.   Assuming that $0 \le k \le n-1$, we solve $dz = d Ev$ for $z \in H_0^1(B)$ and $d \varphi = Ev - z$ for $\varphi \in H_0^1 \Lambda^{k-1}(B)$, as in Lemma \ref{lem3-0}.  The necessary compatibility conditions may be easily verified using $d \circ d= 0$, $d \trace = \trace d$, and $d Ev = dz$.    Restricting $\varphi$ and $z$ to $\Omega$, we obtain \eqref{eq3-8-a} by employing the boundedness of $E$ along with Lemma \ref{lem3-0}.  In the case $k=n$, we let $z =  (|B|^{-1} \int_B Ev) {\rm vol}$, where ${\rm vol}$ is the volume form.  We then have $\int_B (Ev-z) = 0$, and proceeding by solving $d \varphi = Ev-z$ as above completes the proof in this case also.  

In the case $v \in \mathring{H} \Lambda^k(\Omega)$, Lemma \ref{lem3-0-b} may be obtained directly from Lemma \ref{lem3-0} when $\Omega$ is simply connected by applying the procedure in the previous paragraph with $B=\Omega$.   The general case follows by a covering argument.  Let $\{\Omega_i\}$ be a finite open covering of $\Omega$ such that $\Omega \cap \Omega_i$ is Lipschitz for each $i$, and let $\{\chi_i\}$ be a partition of unity subordinate to $\{\Omega_i\}$.  When $ 0 \le k \le n-1$, we first solve $d z_i= d (\chi_i v)$ for $z_i  \in H_0^1\Lambda^k (\Omega_i \cap \Omega)$ and let $z = \sum_{\Omega_i} z_i \in H_0^1 \Lambda^k (\Omega)$.  A simple calculation shows that the compatibility conditions of Lemma \ref{lem3-0} are satisfied, and in addition $dz = dv$ since $\sum_{\Omega_i} \chi_i=1$.  Similarly, we solve $d \varphi_i = \chi_i (v-z)$ for $\varphi_i \in H_0^1 \Lambda^{k-1}(\Omega_i)$ and set $\varphi = \sum_{\Omega_i} \varphi_i$.  In the case $k=n$, let $z_i = \left (  |\Omega_i\cap \Omega|^{-1} \int_{\Omega_i \cap \Omega} \chi_i v \right ) {\rm vol} $ and let $\varphi_i \in H_0^1 \Lambda^{n-1} (\Omega \cap \Omega_i)$ solve $d \varphi_i = \chi_i v - z_i$.  Setting $\varphi = \sum_{\Omega_i} \varphi_i$, $z = v-d \varphi$, and employing Lemma \ref{lem3-0} completes the proof.  
\end{proof}

Our next lemma combines the regular decomposition result of Lemma \ref{lem3-0-b} with a commuting quasi-interpolant in order to obtain approximation results suitable for a posteriori error estimation.  Sch\"oberl defined such an interpolant in \cite{Sch01} for the classical three-dimensional de Rham complex and extended his results to include essential boundary conditions in \cite{Sch08}.  We generally follow Sch\"oberl's construction here, although our notation appears quite different since we use the unified notation of differential forms.  Sch\"oberl instead employed classical notation, which makes the necessary patterns clear and perhaps more concrete but also requires a different definition of the interpolant for each degree of forms ($H^1$, $H(\curl)$, $H(\Div)$, and $L_2$).  Our development also has many similarities to that of Christiansen and Winther in \cite{ChWi08}, who develop a commuting projection operator for differential forms.  Their operator is however global and not suitable for use in a posteriori error estimation.    

In \cite{Sch08} Sch\"oberl defines and analyzes a regular decomposition of the difference between a test function and its interpolant over local element patches for the three-dimensional de Rham complex, whereas we first carry out a global regular decomposition and then interpolate.  The only advantage of Sch\"oberl's approach in the context of a posteriori error estimation seems to be localization of domain-dependent constants, which appears to be mainly a conceptual advantage since the constants are not known. Also, commutativity of the quasi-interpolant is not necessary for the proof of a posteriori error estimates, although it may simplify certain arguments.  See \cite{CXZ12} for proofs of a posteriori estimates for Maxwell's equations that employ a global regular decomposition but not a commuting interpolant.  Commutativity may however be useful in other contexts.  For example, one may  modify the proofs of quasi-optimality of an AFEM for Maxwell's equations in \cite{ZCSWX11} to use a global regular decomposition and a quasi-interpolant instead of a local regular decomposition of the interpolation error so long as the interpolant commutes.




\begin{lemma}\label{lem3-1}
Assume that $v \in H \Lambda^k (\Omega)$ with $\|v\|_H \le 1$.  Then there exists an operator $\Pi_h^k:L_2 \Lambda^k(\Omega) \rightarrow V_h^k$ such that $d^{k+1} \Pi_h^k = \Pi_h^{k+1} d^k$, and in addition the following hold.    If $k=0$, $H \Lambda^0=H^1$ holds and we have
  \begin{align}\label{eq3-9}
  \begin{aligned}
   \sum_{K \in \T_h} \Big [ h_K^{-2} \|v-\Pi_h v\|_{K}^2 & + h_K^{-1} \|\trace (v- \Pi_h v)\|_{\partial K}^2+ |v-\Pi_h v|_{H^1(K)}^2 \Big ] 
 \lesssim 1.  
 \end{aligned}
  \end{align}
 If $1 \le k \le n-1$, there exist $\varphi \in H^1 \Lambda^{k-1}(\Omega)$ and $z \in H^1 \Lambda^{k}(\Omega)$ such that $v=d \varphi + z$, $\Pi_h^k v=d \Pi_h^{k-1} \varphi + \Pi_h^k z$, and for any $K \in \T_h$, 
\begin{align}\label{eq3-8}
\begin{aligned}
 & \sum_{K \in \T_h} \Big [ h_K^{-2}(\|\varphi-\Pi_h \varphi \|_{K}^2 + \|z-\Pi_h z \|_{K}^2)
\\ &  ~~~+ h_K^{-1} (  \|\trace (\varphi -\Pi_h \varphi) \|_{\partial K}^2 + \|\trace (z- \Pi_h z)\|_{\partial K}^2  )
 \Big ]   
   \lesssim 1 .
   \end{aligned}
  \end{align}
In the case $k=n$, the space $H\Lambda^k(\Omega)$ is $L_2 \Lambda^k (\Omega)$, and there exist $\varphi \in H^1 \Lambda^{k-1}(\Omega)$, $z \in L_2 \Lambda^n(\Omega)$ such that $v=d \varphi+z$, $\Pi_h v =d \Pi_h \varphi + \Pi_h z$, and 
\begin{align}\label{eq3-9-a}
\begin{aligned}
 & \sum_{K \in \T_h} \Big [ h_K^{-2}(\|\varphi-\Pi_h \varphi \|_{L_2\Lambda^{k-1} (K)}^2 + \|z-\Pi_h z \|_{L_2\Lambda^{k} (K)}^2)
 \\ &~~~+ h_K^{-1}   \|\trace (\varphi -\Pi_h \varphi) \|_{L_2(\partial K)}^2
 \Big ]    \lesssim 1.
   \end{aligned}
  \end{align}
Assume that $1 \le k \le n$ and $\phi \in H \Lambda^{k-1}(\Omega)$ with $\|\phi\|_H \le 1$.  Then there exists $\varphi \in H^1 \Lambda^{k-1} (\Omega)$ such that $d \varphi =d \phi$, $\Pi_h d \phi= d \Pi_h \phi=d \Pi_h \varphi$, and
\begin{align} \label{eq3-9-b}
\sum_{K \in \T_h} \Big [ h_K^{-2} \|\varphi-\Pi_h \varphi\|_K^2 + h_K^{-1}\|\trace (\varphi -\Pi_h \varphi)\|_{\partial K}^2 \Big ] \lesssim 1.
\end{align}
Finally, the above statements hold with $H \Lambda^k(\Omega)$ replaced by $\mathring{H}_0 \Lambda^k(\Omega)$ and $H^1\Lambda^k(\Omega)$ replaced by $H_0^1 \Lambda^k(\Omega)$, and in this case $\Pi_h : \mathring{H} \Lambda^k(\Omega) \rightarrow \mathring{V}_h^k$.  
 \end{lemma}

\begin{proof}
 Let $\Pi_h= I_h R_h^\epsilon$, where following \cite{Sch08, ChWi08} $I_h$ is the canonical interpolant for smooth forms and $R_h^\epsilon$ is a smoothing operator with smoothing parameter $\epsilon$ which we detail below.  (Note that the construction in \cite{ChWi08} involves a further operator $J_h^\epsilon$, which is the inverse of $\Pi_h$ applied to $V_h^k$.  $J_h^\epsilon$ is non-local and thus not suitable for use here.)   We conceptually follow \cite{Sch08} in our definition of $R_h^\epsilon$ but also use some technical tools from and more closely follow the notation of \cite{ChWi08}.  We omit a number of details, so a familiarity with these works would be helpful to the reader. 

Following \cite{ChWi08}, there is a Lipschitz-continuous vector field $X(x)$ defined on a neighborhood of $\Omega$ such that $X(x) \cdot \vec{n}(x)>0$ for all outward unit normals $\vec{n}(x)$, $x \in \partial \Omega$. Let $g_h(x)$ be the natural Lipschitz-continuous mesh size function.  There is then $\delta>0$ so that for $\epsilon>0$ sufficiently small depending on $\Omega$, $B_{\epsilon}(x + \delta \epsilon g_h(x) X(x)) \subset \mathbb{R}^n \setminus \Omega$ for all $x \in \overline{\Omega}$ with ${\rm dist}(x, \partial \Omega)<\epsilon g_h(x)$, and $B_{\epsilon}(x - \delta \epsilon g_h(x) X(x)) \subset \Omega$ for all $x \in \partial \Omega$.  Let $\chi_V$ be the standard piecewise linear ``hat'' function associated to the vertex $V$, and let $\mathcal{N}_\partial$ be the set of boundary vertices.  Extending $\chi_V$ and $g_h$ by reflection over $\partial \Omega$ via $X$ (cf. \cite{ChWi08}), we define $\Phi_D(x) = x + \sum_{V \in \mathcal{N}_\partial} \chi_V(x) \delta \epsilon g_h(x) X(x)$ and $\Phi_N(x) = x-\sum_{V \in \mathcal{N}_\partial} \chi_V(x) \delta \epsilon g_h(x) X(x)$.  

Following \cite{Sch08}, we associate to each vertex $V$ in $\T_h$ a control volume $\Omega_{V}$.  Let $\Omega_V$ be the ball of radius $\epsilon g_h(V)/2$ centered at $V$ if $V$ is an interior  or Dirichlet boundary node, and at $\Phi_N(V)$ if $V$ is a boundary Neumann node.  Also let $f_V \in \mathbb{P}_n$ satisfy
\begin{align}
\label{eq3-11-a}
\int_{\Omega_V} p(x) f_V(x) \d x = p(V), ~~p \in \mathbb{P}_n.
\end{align}

\setlength{\unitlength}{.75cm}
\begin{figure}[h]
\begin{center}
\includegraphics[height=2in]{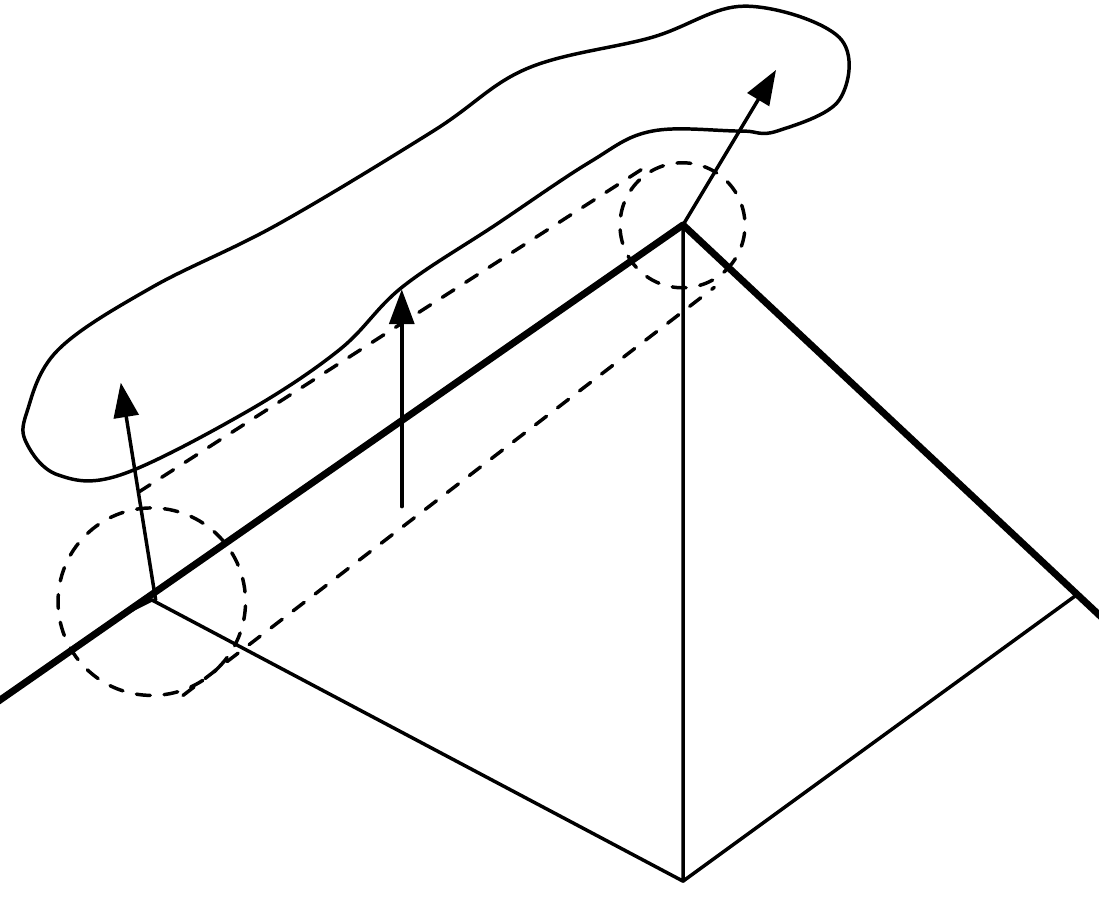}
\includegraphics[height=2in]{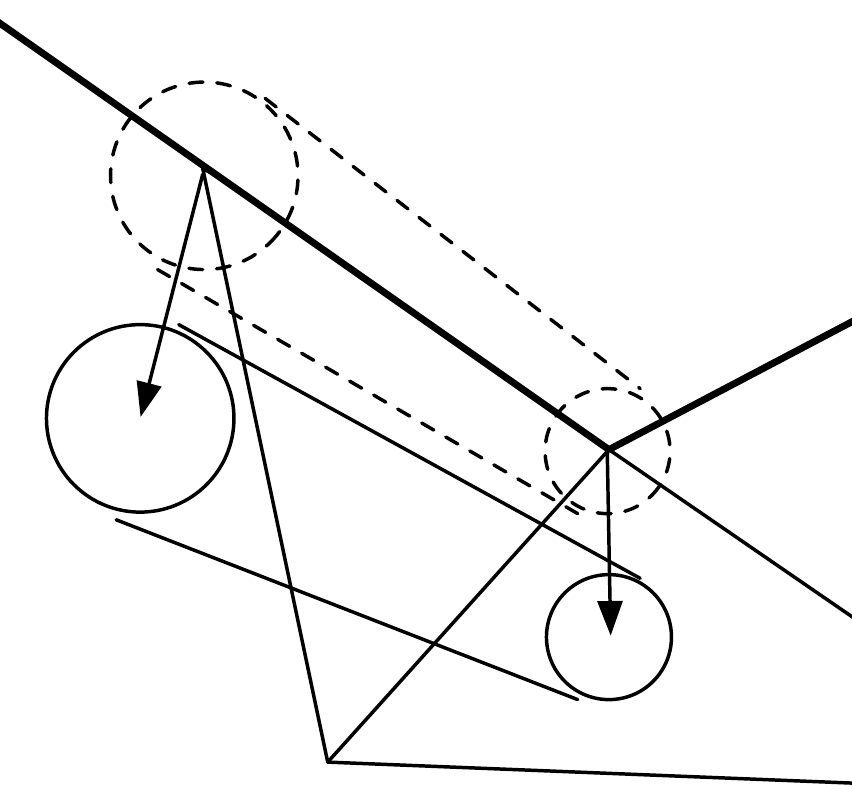}
\put(-14.7,1.7){\parbox[b]{2cm}{$V$}}
\put(-15,4){\parbox[b]{2cm}{$\Phi_D(V)$}}
\put(-14.7,1.1){\parbox[b]{2cm}{$\Omega_V$}}
\put(-5.6,5.5){\parbox[b]{2cm}{$V$}}
\put(-7.55, 3.2){\parbox[b]{2cm}{$\Omega_V$}}
\put(-6.75,2.95){\parbox[b]{2cm}{$\Phi_N(V)$}}
\end{center}
\caption{Schematic of transformations and control volumes for the Dirichlet (left) and Neumann (right) cases.}
\label{fig1}
\end{figure}
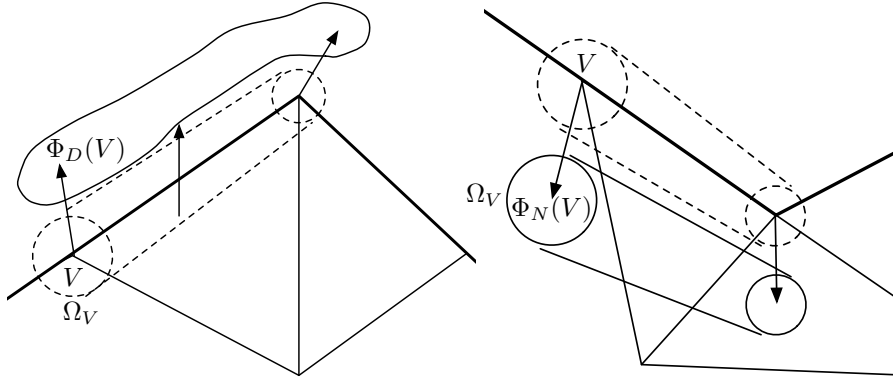

Given $K \in \T_h$, let $\{V_1, ..., V_{n+1}\}$ be the vertices with associated barycentric coordinates $\{\lambda_1(x), ..., \lambda_{n+1}(x)\}$ for $x \in K$.  Let $\hat{x}(x, y_1,...., y_{n+1}) = \sum_{i=1}^{n+1} \lambda_i(x) y_i = x + \sum_{i=1}^{n+1} \lambda_i(x) (y_i -V_i)$ for $y_i \in \Omega_{V_i}$.  We also denote by $\omega$ the extension of $\omega$ by 0 to $\mathbb{R}^n$ in the case of Dirichlet boundary conditions, and the extension of $\omega$ to a neighborhood of $\Omega$ by smooth reflection via $X$ in the case of Neumann boundary conditions (cf. \cite{ChWi08}).  Letting $\tilde{x}=\hat{x}$ in the case of natural boundary conditions and $\tilde{x}=\Phi_D \circ \hat{x}$ in the case of essential boundary conditions, we define 
\begin{align}
\label{eq3-12}
(R_h^\epsilon \omega)_x = \int_{\Omega_{V_1}} \cdots \int_{\Omega_{V_{n+1}}} f_{V_1} (y_1) \cdots f_{V_{n+1}}(y_{n+1})  (\tilde{x}^*\omega )_x  \d y_{n+1} \cdots \d y_1.
\end{align}
Commutativity of $R_h^\epsilon$ and thus of $\Pi_h$ immediately follows as in \cite{Sch01, Sch08, ChWi08}.  In addition, $\omega=0$ on $\mathbb{R}^n \setminus \Omega$ implies that $(R_h^\epsilon \omega)_x=0$ for $x \in \partial \Omega$ in the case of essential boundary conditions, since in this case ${\rm dist}(\hat{x}(x), \partial \Omega) <  g_h(x) \epsilon$ for all $(y_1,...,y_{n+1}) \in \Omega_{V_1} \times \cdots \times \Omega_{V_{n+1}}$ so that $\Phi_D \circ \hat{x}(x, y_1, ..., y_{n+1}) \in \mathbb{R}^n \setminus \Omega$.   

We now establish that $R_h^\epsilon$ preserves constants locally on elements sufficiently removed from $\partial \Omega$.  This implies the same of $\Pi_h$ since $I_h$ preserves constants.  Let $\omega$ be a constant $k$-form, and let $x \in \Omega$ lie in $K\in \T_h$ which in the Neumann case satisfies $\overline{K} \cap \partial \Omega = \emptyset$ and which in the Dirichlet case satisfies $\overline{\omega}_K \cap \partial \Omega = \emptyset$.   Noting that now $\tilde{x} =\hat{x}$ on $\omega_K$, computing that $D  \hat{x}(x, y_1,..., y_{n+1}) = I + \sum_{i=1}^{n+1} (y_i-V_i) \otimes \nabla \lambda_i(x)$, applying \eqref{eq3-1-aa}, and applying the multilinearity of $\omega$, we find that for $n$-vectors $v_1,...,v_k$
\begin{align}
\label{eq3-13}
\begin{aligned}
 (& R_h^\epsilon    \omega)_x   (v_1,...,v_k) 
 = \int_{\Omega_{v_1}} \cdot \cdot \int_{\Omega_{V_{n+1}}} f_{V_1} (y_1) \cdots f_{V_{n+1}}(y_{n+1}) 
 \\ &  \times \omega\left (v_1 + \sum_{i=1}^{n+1} (\nabla \lambda_i(x) \cdot v_1) (y_i -V_i),...,v_k + \sum_{i=1}^{n+1} (\nabla \lambda_i(x) \cdot v_k) (y_i-V_i) \right )   
 \\ & \times  \d y_{n+1} \cdots \d y_1  
\\ & = \omega(v_1,..., v_k) + \Phi.
\end{aligned}
\end{align}
Here $\Phi$ consists of a sum of constants multiplying terms of the form $\omega(z_1, ..., z_k)$, where $z_i = y_j-V_j$ for at least one entry vector $z_i$ and for some $1 \le j \le n+1$.  \eqref{eq3-11-a} then implies that $ \int_{\Omega_{v_1}} \cdots \int_{\Omega_{V_{n+1}}} f_{V_1} (y_1) \cdots f_{V_{n+1}}(y_{n+1}) \Phi  \d y_{n+1} \cdots \d y_1 = 0$.  Observing that $\int_{\Omega_V} f_V (y) \d y =1$ completes the proof that $R_h^\epsilon \omega (v_1,..., v_k) = \omega(v_1,..., v_k)$ in the case that $x$ lies in an interior element.  

Arguments as in \cite{ChWi08, Sch01, Sch08} show that for $\omega \in L_2 \Lambda^k$, $\|\Pi_h \omega\|_K \lesssim \|\omega\|_{K^*}$, where $K^*=\omega_K$ if $K$ is an interior element and $K^*= \omega_K \cup \{x \in \mathbb{R}^n: {\rm dist}(x,K) \lesssim \epsilon h_K\}$ otherwise; cf. Figure 4.2 of \cite{ChWi08}.  In the latter case the definition of the extension of $\omega$ to $\mathbb{R}^n \setminus \Omega$ as either 0 or the pullback of $\omega$ under a smooth reflection implies that the values of $\omega$ on $K^*$ in fact depend only on values of $\omega$ on $\omega_K$ so long as $\epsilon$ is sufficiently small, which in turn implies that $\|\Pi_h \omega\|_K \lesssim \| \omega\|_{K^*} \lesssim \|\omega\|_{\omega_K}$.  For $K$ not abutting $\partial \Omega$ we may combine the above properties of $\Pi_h$ with the Bramble-Hilbert Lemma (cf. \cite{BS08}) in order to yield $h_K^{-1} \|z-\Pi_h z\|_{K} +|z-\Pi_h z|_{H^1 \Lambda^k (K)} \lesssim |z|_{H^1\Lambda^k (\omega_K)}$ for $z \in H^1 \Lambda^k(\Omega)$.  A standard scaled trace inequality for $z \in H^1 \Lambda^k (K)$ reads $\|\trace(z-\Pi_h z)\|_{\partial K} \lesssim h_K^{-1/2} \|z-\Pi_h z\|_{K}+h_K^{1/2} |z -\Pi_h z|_{H^1 \Lambda^k (K)}$.  Combining these inequalities with the finite overlap of the patches $\omega_K$ implies \eqref{eq3-8}, \eqref{eq3-9}, and \eqref{eq3-9-a} modulo boundary elements. 

When $V^k = H \Lambda^k(\Omega)$ and $x \in \overline{K}$ with $\overline{K} \cap \partial \Omega \neq \emptyset$, the above argument that constants are locally preserved and a standard Bramble-Hilbert argument applies holds so long as the convex hull of $\Omega_{V_1},..., \Omega_{V_{n+1}}$ lies in $\Omega$ for the vertices $V_1,.., V_{n+1}$ of $K$.  This is true if $h_K \le h_0$ for some $h_0$ depending on $\Omega$.  To prove this, let $V_1,..., V_k$ be the vertices of $K$ lying on $\partial \Omega$.  
Our assumptions imply that $\Omega_{V_i} \subset \Omega$, $i=1,...,k$.  We must show that the convex hull of  $\Omega_{V_1},..., \Omega_{V_k}$ also is a subset of $\Omega$.  This set is the union of all balls $B_{\epsilon g_h(y)/2} (y)$, where $y=\sum_{i=1}^k \lambda_i \Phi_N(V_i)$ for $0 \le \lambda_i \le 1$ satisfying $\sum_{i=1}^k \lambda_i=1$.  Fixing such a $y$,  let $\tilde{y}=\sum_{i=1}^k \lambda_i V_i \in \partial K \cap \partial \Omega$.  Then ${\rm dist} (\Phi_N(\tilde{y}), \partial \Omega) \ge \epsilon g_h(\tilde{y})$.  But $|\Phi_N (\tilde{y})-y| \le \epsilon g_h(y) h_K \|D X\|_{L_\infty}$, so that $B_{\epsilon g_h(y)/2}(y) \subset \Omega$ so long as $h_K \|DX\|_{L_\infty(K)} \le \frac{1}{2}$.  We thus take $h_0 = \frac{1}{2 \|DX\|_{L_\infty(\Omega)}}$. Thus the results of Lemma \ref{lem3-1} follow as above when $h_K \le h_0$.  If the convex hull of $\Omega_{V_1},..., \Omega_{V_{n+1}}$ does not lie in $\Omega$, then the integral \eqref{eq3-12} may sample values of $\omega_x$ for $x \notin \Omega$.  Extension by pullback of a reflection does not preserve constant forms, and so in this case $\Pi_h$ does not necessarily preserve constants locally.  The operator $\Pi_h$ is however still locally $L_2$ bounded, and since in this case we have $h_K >h_0$ we may think of $h_K$ as merely a constant depending on $\Omega$ and thus still obtain the approximation estimates of Lemma \ref{lem3-1} after elementary manipulations.  

In the case where $V^k = \mathring{H} \Lambda^k(\Omega)$ and $x$ lies in a patch $\omega_K$ with $\overline{K} \cap \partial \Omega \neq \emptyset$, $\Pi_h$ does not preserve constants locally, but is still locally $L_2$ bounded.  Because we in this case apply $\Pi_h$ to forms $z, \varphi \in H_0^1$ (that is, to forms which are identically 0 on $\partial \Omega$), we may apply a Poincar\'e inequality in conjunction with elementary manipulations in order to obtain \eqref{eq3-8}, \eqref{eq3-9}, and \eqref{eq3-9-a} in this case as well.

 \eqref{eq3-9-b} follows by a similar argument, that is, by extending $\phi$ $H$-continuously to a ball $B$ in the Neumann case, solving the boundary value problem $d \varphi = d \phi$ on $B$ or $\Omega$ as appropriate, employing the $H^1$ regularity result of Lemma \ref{lem3-0}, and then applying properties of $\Pi_h$.  
\end{proof}


\section{Reliability of a posteriori error estimators}\label{sec_reliability}

In this section we define and prove the reliability of a posteriori estimators.  We will establish a series of lemmas bounding in turn each of the terms in \eqref{eq2-23}.  Below we denote by $\llbracket \chi \rrbracket$ the jump in a quantity $\chi$ across an element face $e$.  In case $e \subset \partial \Omega$, $\llbracket \chi \rrbracket$ is simply interpreted as $\chi$.  All of our results and discussion below are stated for the case of natural boundary conditions; results for essential boundary conditions are the same with the modification that edge jump terms are taken to be 0 on boundary edges.  

\subsection{Reliability:  Testing with $ \tau \in H \Lambda^{k-1}$}

\begin{lemma}\label{lem4-1}
Given $K \in \T_h$, let
\begin{align} \label{eq4-1}
\eta_{-1}(K)= \left \{ \begin{array}{l}   0 \hbox{ for } k=0,
\\  h_K \|\sigma_h -\delta u_h\|_K+ h_K^{1/2}  \|\llbracket \trace \star  u_h \rrbracket \|_{\partial K} \hbox{ for } k=1,
\\  h_K (\|\delta \sigma_h \|_K + \|\sigma_h -\delta u_h\|_K) 
\\  ~~~~~+ h_K^{1/2} ( \|\llbracket \trace \star \sigma_h \rrbracket \|_{\partial K} + \|\llbracket \trace \star  u_h \rrbracket \|_{\partial K}) \hbox{ for } 2 \le k \le n.
\end{array} \right . 
\end{align}
Let $(\sigma, u, p)$ be the weak solution to \eqref{eq3-4}-\eqref{eq3-6}, let $(\sigma_h, u_h, p_h)$ be the corresponding finite element solution having errors $(e_\sigma, e_u, e_p)$, and assume $\tau \in H \Lambda^{k-1}(\Omega)$ with $\|\tau\|_{H \Lambda^{k-1}(\Omega)} \le 1$.  Then
\begin{align} \label{eq4-2}
|\la e_\sigma, (\tau-\Pi_h \tau) \ra- \la d (\tau-\Pi_h \tau), e_u \ra | \lesssim \left (\sum_{K \in \T_h} \eta_{-1}(K)^2 \right )^{1/2}.
\end{align}
\end{lemma}
\begin{proof}
If $k=0$, then $\tau$ is vacuous and so the term above disappears.  We next consider the case $2 \le k \le n$.  Using Lemma \ref{lem3-1}, we write $\tau=d \varphi + z$, where $\varphi \in H^1 \Lambda^{k-2}(\Omega)$ and $z \in H^1 \Lambda^{k-1}(\Omega)$.  Since $\Pi_h \tau =d \Pi_h \varphi + \Pi_h z$ and $d \circ d=0$, we have $d (\tau-\Pi_h \tau)=d (z-\Pi_h z)$.  Thus using the first line of \eqref{eq2-2} and the integration by parts formula \eqref{eq3-3} on each element $K \in \T_h$, we have
\begin{align} \label{eq4-3}
\begin{aligned}
-\la e_\sigma, &  \tau-\Pi_h \tau \ra + \la d(\tau-\Pi_h \tau), e_u \ra = \la \sigma_h, \tau-\Pi_h \tau \ra - \la d(\tau-\Pi_h \tau), u_h \ra 
\\ & = \sum_{K \in \T_h} \la \sigma_h, d (\varphi - \Pi_h \varphi) + z-\Pi_h z \ra_K  - \la d(z-\Pi_h z), u_h \ra_K
\\ & = \sum_{K \in \T_h} \la \delta \sigma_h, \varphi -\Pi_h \varphi \ra_K + \int_{\partial K} \trace (\varphi - \Pi_h \varphi) \wedge \trace \star  \sigma_h 
\\ & ~~~~+ \la \sigma_h -\delta u_h, z-\Pi_h z\ra_K -\int_{\partial K} \trace (z-\Pi_h z) \wedge \trace \star u_h.
\end{aligned}
\end{align}
Note next that $\trace(z-\pi_h z)$ is single-valued on an edge $e=K_1 \cap K_2$ , since $z \in H^1 \Lambda^k$ and $\Pi_h z \in H \Lambda^k$.  $\trace \star u_h$ on the other hand is different depending on whether it is computed as a limit from $K_1$ or from $K_2$, and we use $\llbracket \trace \star u_h \rrbracket$ to denote its jump ($\llbracket \trace \star u_h \rrbracket = \trace \star u_h$ on $\partial \Omega$).  A similar observation holds for $\trace (\varphi -\Pi_h \varphi)$ and $\trace \star \sigma_h$.  Let $\E_h$ denote the set of faces ($n-1$-dimensional subsimplices) in $\T_h$, and let $\star_{\partial K}$ denote the Hodge star on $\Lambda^j (\partial K)$ (with $j$ determined by context).  We then have using \eqref{eq3-1-a} and the fact that the Hodge star is an $L_2$-isometry that
\begin{align}\label{eq4-4}
\begin{aligned}
\sum_{K \in \T_h} \int_{\partial K} \trace (\varphi- \Pi_h \varphi) & \wedge\trace \star \sigma_h =  \sum_{e \in \E_h} \la \star_{\partial K} \trace (\varphi-\Pi_h \varphi), \llbracket \trace \star \sigma_h \rrbracket \ra
\\ & \lesssim \sum_{K \in T_h} \|\trace (\varphi -\Pi_h \varphi)\|_{\partial K} \|\llbracket \trace \star \sigma_h \rrbracket \|_{\partial K}.
\end{aligned}
\end{align}
Similarly manipulating the other boundary terms in \eqref{eq4-3} and employing \eqref{eq3-8} yields
\begin{align}\label{eq4-5}
\begin{aligned}
\la \sigma_h, & \tau-\Pi_h \tau \ra - \la d (\tau-\Pi_h \tau), u_h \ra
\\ & \lesssim \sum_{K \in \T_h} \eta_{-1}(K) \Big [ h_K^{-1} (\|z-\Pi_h z\|_K + \|\varphi -\Pi_h \varphi\|_K) 
\\ & ~~+h_K^{-1/2}(\|\trace(z-\Pi_h z)\|_{\partial K} + \|\trace (\varphi -\Pi_h \varphi)\|_{\partial K}) \Big ]
\\ & \lesssim   \Big  ( \sum_{K \in \T_h} \eta_{-1}(K)^2 \Big)^{1/2} 
\times \Big ( \sum_{K \in \T_h} \Big [ h_K^{-2}(\|\varphi-\Pi_h \varphi \|_{K}^2 + \|z-\Pi_h z \|_{K}^2)
 \\ &~~+ h_K^{-1} (  \|\trace (\varphi -\Pi_h \varphi) \|_{\partial K}^2 + \|\trace (z- \Pi_h z)\|_{\partial K}^2  )\Big ] \Big )
\\ &  \lesssim   \Big  ( \sum_{K \in \T_h} \eta_{-1}(K)^2 \Big)^{1/2}.
  \end{aligned}
\end{align}
Thus the proof is completed for the case $2 \le k \le n$.  

For the case $k=1$ we have by definition that $z=\tau \in H \Lambda^0(\Omega)= H^1(\Omega)$. Thus the proof proceeds as above but with terms involving $\delta \sigma_h$ and $\trace \star \sigma_h$ omitted.
\end{proof}

\subsection{Reliability:  Testing with $v \in H \Lambda^k$} \label{sec4-2}

In our next lemma we bound the term $\la f - d \sigma_h -p_h, v - \Pi_h v \ra - \la d  u_h, d (v- \Pi_h v) \ra$ from \eqref{eq2-23}.   Before doing so we note that $\H^n$ and $\H_h^n$ are always trivial, so in this case $p=p_h=0$.  We however leave the harmonic term $p_h$ in our indicators even when $k=n$ for the sake of consistency with the other cases.

\begin{lemma}\label{lem4-2}  
Let $K \in \T_h$, and assume that $f \in H^1 \Lambda^k (K)$ for each $K \in \T_h$.  Let 
\begin{align} \label{eq4-20}
\eta_0(K)= \left \{ \begin{array}{l}  h_K \|f-p_h - \delta d u_h\|_K+ h_K^{1/2} \|\llbracket \trace \star d u_h \rrbracket \|_{\partial K} \hbox{ for } k=0,
\\ h_K ( \|f-d \sigma_h - p_h -\delta d u_h \|_K + \|\delta (f-d\sigma_h -p_h)\|_K) 
\\  ~~~~~+ h_K^{1/2} ( \|\llbracket \trace \star d u_h \rrbracket  \|_{\partial K} + \|\llbracket \trace \star (f-d \sigma_h -p_h) \rrbracket \|_{\partial K}) 
\\ ~~~~~~~~~~\hbox{ for } 1 \le k \le n-1
\\ \|f-d \sigma_h-p_h\|_K \hbox{ for } k =n.
\end{array} \right . 
\end{align}
Under the above assumptions on the regularity of $f$ and with all other definitions as in Lemma \ref{lem4-1} above, we have for any $v \in H \Lambda^k (\Omega)$ with $\|v\|_{H \Lambda^k (\Omega)} \le 1$
\begin{align} \label{eq4-24}
 \la f - d \sigma_h -p_h, v - \Pi_h v \ra - \la d  u_h, d (v- \Pi_h v) \ra  \lesssim \left (  \sum_{K \in \T_h} \eta_0(K)^2  \right )^{1/2}.
 \end{align}
\end{lemma}

\begin{proof}
For $k=n$, the term $\la d u_h, d(v-\Pi_h v) \ra$ is vacuous, and Galerkin orthogonality implies that
$\la f-d\sigma_h-p_h, v-\Pi_h v \ra   = \la f-d \sigma_h -p_h, v \ra \lesssim \left (  \sum_{K \in \T_h} \eta_0(K)^2\right )^{1/2}$.

For $0 \le k \le n-1$, noting that $d(v-\Pi_h v)=d (z-\Pi_h z + d (\varphi -\Pi_h \varphi))=d (z-\Pi_h z)$ and integrating by parts yields
\begin{align}\label{eq4-26}
\begin{aligned}
\la d u_h, d(v-\Pi_h v) \ra & =  \la d u_h, d(z-\Pi_h z) \ra
\\ &  = \sum_{K \in \T_h} \la \delta d u_h, z-\Pi_h z \ra + \int_{\partial K} \trace (z-\Pi_h z) \wedge \trace \star d u_h.
\end{aligned}\end{align}
For $k=0$ both $\varphi$ and $\sigma_h$ are vacuous, so we may complete the proof by employing \eqref{eq4-26} and proceeding as in \eqref{eq4-4} and \eqref{eq4-5} to obtain
\begin{align}\label{eq4-27}
\begin{aligned}
 \la f - & d \sigma_h -p_h, v - \Pi_h v \ra  - \la d  u_h, d (v- \Pi_h v) \ra 
 \\ & =\sum_{K \in \T_h} \la f-p_h -\delta d u_h, v-\Pi_h v \ra - \int_{\partial K} \trace(v-\Pi_h v) \wedge \trace \star d u_h
 \\ & \lesssim \Big (\sum_{K \in \T_h} \eta_0(K)^2 \Big )^{1/2} \|v\|_{H^1(\Omega)} \lesssim  \Big (\sum_{K \in \T_h} \eta_0(K)^2 \Big )^{1/2}.
\end{aligned}
\end{align}

For $1 \le k \le n-1$, we write $v=d\varphi+z$ as in Lemma \ref{lem3-1} and employ \eqref{eq4-26} to find
\begin{align} \label{eq4-28}
\begin{aligned}
& \la  f -  d \sigma_h -p_h, v - \Pi_h v \ra  - \la d  u_h, d (v- \Pi_h v) \ra 
\\ & = \Big [  \la f-d\sigma_h -p_h, d (\varphi-\Pi_h \varphi) \ra_K \Big ]
\\&  ~~ + \Big [ \sum_{K \in \T_h} \la f-d \sigma_h -p_h -\delta d u_h, z- \Pi_h z \ra_K 
\int_{\partial K} \trace (z-\Pi_h z) \wedge \trace  \star du_h\Big ] 
\\ & \equiv [I]+[II].
\end{aligned}
\end{align}
The term $II$ above may be manipulated as in \eqref{eq4-27} above in order to obtain
\begin{align}\label{eq4-29}
II \lesssim \Big ( \sum_{K \in \T_h}  h_K^2 \|f-d\sigma_h -p_h -\delta d u_h\|_K^2 + h_K\|\llbracket \trace \star d u_h \rrbracket \|_{\partial K} \Big ) ^{1/2}. 
\end{align}

We now turn our attention to the term $I$.  Integrating by parts while proceeding as in \eqref{eq4-4} and \eqref{eq4-5} yields
\begin{align}\label{eq4-30}
\begin{aligned}
I & = \sum_{K \in \T_h} \la \delta (f-d\sigma_h -p_h), \varphi - \Pi_h \varphi \ra_K 
\\ & ~~~~+ \int_{\partial K} \trace (\varphi - \Pi_h \varphi)\wedge \trace \star (f-d \sigma_h - p_h).
\\ & \lesssim  \Big ( \sum_{K \in \T_h} h_K^2 \|\delta (f-d \sigma_h -p_h)\|_K^2 + h_K \|\llbracket \trace \star(f-d\sigma_h -p_h) \rrbracket \|_{\partial K}^2 \Big ) ^{1/2}.  
\end{aligned}
\end{align}
Combining \eqref{eq4-29} and \eqref{eq4-30} yields \eqref{eq4-24} for $1 \le k \le n-1$, completing the proof.  
\end{proof}

We finally remark on an important feature of our estimators.  The term $h_K \|\delta (f-d \sigma_h-p_h)\|_K+ h_K^{1/2} \|\llbracket \trace \star (f-d \sigma_h -p_h ) \rrbracket \|_{\partial K}$ is in a sense undesirable because it requires higher regularity of $f$ than merely $f \in L_2$.  In particular, evaluation of the first term requires $f \in H^* \Lambda^k(K)$  for each $K \in \T_h$, and evaluation of the trace term requires $\trace \star f \in L_2\Lambda (\partial K)$ for each $K$.  (Note however that $f$ is not included in the jump terms if $f \in H^* \Lambda^k(\Omega)$.)  Both relationships are implied by $f \in H^1 \Lambda^k(K)$, $K \in \T_h$, so we simply make this assumption.  

In order to understand why such terms appear, note that the Hodge decomposition of $f$ reads $f=d \sigma + p + \delta d u$.  The first two terms $d \sigma+p$ are approximated {\it directly in $L_2$} in the mixed method by $d \sigma_h + p_h$, while the latter term $\delta d u$ is only approximated {\it weakly in a negative order norm} (roughly speaking, in the space dual to $H \Lambda_k^*$) in the mixed method.  In our indicators, $\|(d \sigma+p)-(d\sigma_h + p_h)\|_K$ is thus a naturally scaled and efficient residual for the mixed method, but $\|\delta d u-\delta d u_h\|_K$ is one Sobolev index too strong.  The latter term should instead be multiplied by a factor of $h_K$ in order to mimic a norm with Sobolev index $-1$, as in the term $h_K \|f-d \sigma_h -p_h -\delta d u_h\|_{K}$ appearing in $\eta_0$.  

This ``Hodge imbalance'' implies that it is necessary to carry out a Hodge decomposition of $f$ in order to obtain error indicators that are correctly scaled for all variables. When this decomposition is unavailable a priori, the Hodge decomposition must be carried out weakly in order to obtain a {\it computable} and {\it reliable} estimator in which the appropriate parts of the Hodge decomposition of $f$ are scaled correctly.  This is accomplished above.  Since $\delta (\delta d u)=0$, $h_K \|\delta (f-d \sigma_h -p_h)\|_{K} =h_K (\|\delta d e_\sigma+ \delta e_p\|_K)$.  This scales roughly as a Sobolev norm with order $-1$ of $\delta d e \sigma$ and $d e_p$, which in turn scales as the terms $\|d e_\sigma\|+ \|e_p\|$ appearing in the original error we seek to bound.  For an element face $e \in \partial \Omega$, \eqref{eq3-7} along with $\trace \star p=0$ on $\partial \Omega$ imply that $\llbracket \trace \star (f-d\sigma_h-p_h) \rrbracket =\trace \star (d\sigma-d\sigma_h -p_h)$ on $\partial \Omega$. Similarly, for an interior face $e$ we have $\llbracket \trace \star f \rrbracket = \llbracket \trace \star  (d\sigma-d\sigma_h-p_h) \rrbracket$. 

If a partial Hodge decomposition of $f$ is known, it is possible to redefine $\eta_0$ so that only $f \in L_2$ is required.  If $f=d \sigma + \psi$ with $\psi = p + \delta d u$ known a priori, we may replace $h_K \|\delta (f-d \sigma_h-p_h)\|_K+ h_K^{1/2} \|\llbracket \trace \star (f-d \sigma_h -p_h ) \rrbracket \|_{\partial K}$ with $h_K \|\delta p_h \|_K +  \|d \sigma-d \sigma_h \|_K+ h_K^{1/2} \|\llbracket \trace \star p_h \rrbracket \|_{\partial K}$.  If $f = \Theta + \delta d u$ with $\Theta=d \sigma + p$ known a priori, we may instead replace this term with $\|\Theta-d \sigma_h -p_h \|_K$.  We do not assume such a decomposition is known, since it if were one would likely decompose the Hodge Laplace problem into $\B$ and $\B^*$ problems, as described in \cite{ArFaWi2010}.  

 
We finally note that a similar situation occurs in residual-type estimates for the time-harmonic Maxwell problem $\curl \curl u -\omega^2 u=f$, where the elementwise indicators include a term $h_K \|\Div (f + \omega^2 u_h)\|_K + h_K^{1/2} \|(f+\omega^2 u_h)\cdot n\|_{\partial K}$ (cf. \cite{ZCSWX11}).  There however the assumption $\Div f \in L_2$ is natural as it represents the charge density, and to our knowledge the connection of this term with the Hodge decomposition has not been previously explained.

\subsection{Reliability:  Harmonic terms}

Next we turn to bounding the terms in \eqref{eq2-23} related to harmonic forms.  
\begin{lemma}\label{lem4-3}
Given $q_h \in V_h^k$ and $K \in \T_h$, let 
\begin{align}\label{eq4-50}
\eta_\H (K, q_h)=h_K \|\delta q_h \|_K + h_K^{1/2} \|\llbracket \trace \star q_h \rrbracket \|_{\partial K}.
\end{align}
Then if $1 \le k \le n$ and $\phi \in H \Lambda^{k-1}(\Omega)$ with $\|\phi \|_{H \Lambda^{k-1} (\Omega)}=1$,  
\begin{align} \label{eq4-51}
\la q_h , d ( \phi -\Pi_h \phi) \ra \lesssim  \Big ( \sum_{K \in \T_h} \eta_\H (K, q_h)^2 \Big ) ^{1/2}.
\end{align}
Given an orthonormal basis $\{q_1, ..., q_M\}$ for $\H_h^k$, let $\mu_i=\big ( \sum_{K \in \T_h} \eta_\H (K, q_i)^2 \big )^{1/2}$.  Then we additionally have
\begin{align} \label{eq4-52}
{\rm gap}(\H^k, \H_h^k) \lesssim \mu :=  \left (  \sum_{i=1}^M \mu_i^2 \right ) ^{1/2}.  
\end{align}
Finally, if $u_h^\perp \in \Z_h^{k \perp}$,
\begin{align}
\label{eq4-52-a}
\|P_\B u_h^\perp \| \lesssim \Big (\sum_{K \in \T_h} \eta_\H (K, u_h^\perp)^2 \Big)^{1/2}.  
\end{align}
\end{lemma}

\begin{proof}
Let $\varphi \in H^1\Lambda^{k-1}(\Omega)$ boundedly solve $d \varphi= d \phi$, as in \eqref{eq3-9-b} and preceding of Lemma \ref{lem3-1}.  Employing \eqref{eq3-9-b}, integrating by parts as in \eqref{eq4-26}, and proceeding as in \eqref{eq4-27} immediately yields
\begin{align}\label{eq4-54}
\begin{aligned}
\la q_h, d (\phi -\Pi_h \phi) \ra & = \sum_{K \in \T_h} \la \delta q_h, \varphi- \Pi_h \varphi \ra_K + \int_{\partial K} \trace (\varphi -\Pi_h \varphi) \wedge \trace \star q_h
\\ & \lesssim \Big ( \sum_{K \in \T_h} \eta_\H (K, q_h)^2 \Big)^{1/2}.  
\end{aligned}
\end{align}
\eqref{eq4-52} immediately follows from \eqref{eq2-16} and \eqref{eq4-51}, while \eqref{eq4-52-a} follows from \eqref{eq2-20} and \eqref{eq4-51}.
\end{proof}

\subsection{Summary of reliability results}  We summarize our reliability results in the following theorem.

\begin{theorem} \label{t4-4} Assume that $\Omega\subset \mathbb{R}^n$ is a bounded Lipschitz domain of arbitrary topological type.  Let $0 \le k \le n$.  Let $\eta_{-1}$ be as defined in Lemma \ref{lem4-1}, let $\eta_0$ be as defined in Lemma \ref{lem4-2}, and let $\eta_{\H}$ be as defined in Lemma \ref{lem4-3}.  Let also $\{q_1, ... q_M\}$ be an orthonormal basis for $\H_h^k$ and let $\mu$ be as in \eqref{eq4-52}.   Then 
\begin{align}\label{eq4-60}
\begin{aligned}
& \|  e_\sigma\|_{H \Lambda^{k-1}(\Omega)}  + \|e_u\|_{H \Lambda^k (\Omega)} + \|e_p\| 
\\ &~~~~ \lesssim \Big ( \sum_{K \in \T_h} \eta_{-1}(K)^2 + \eta_0(K)^2+ \eta_\H(p_h)^2 \Big)^{1/2} +  \mu  \|u_h\|.  
\end{aligned}
\end{align}
 Let also $u_h^\perp$ be the projection of $u_h$ onto $\Z_h^{k \perp}$.  Then the term $\mu \|u_h\|$ in \eqref{eq4-60} may be replaced by 
 \begin{align}\label{eq4-61}
\mu \Big (  \sum_{K \in \T_h} \eta_\H (K, u_h^\perp)^2 \Big )^{1/2} + \mu^2 \|u_h\|.
  \end{align}
\end{theorem}

\begin{proof}
The four terms on the right hand side of \eqref{eq2-23} may be bounded by employing Lemma \ref{lem4-1}, Lemma \ref{lem4-2}, Lemma \ref{lem4-3}, and once again Lemma \ref{lem4-3}, respectively.
\end{proof}


\section{Efficiency of a posteriori error estimators}\label{sec_efficiency}

We consider efficiency of the various error indicators employed in \S\ref{sec_reliability} in turn.  Before doing so, we provide some context for our proofs along with some basic technical tools.

Efficiency results for residual-type a posteriori error estimators such as those we employ here are typically proved by using the ``bubble function'' technique of Verf\"urth \cite{Ver89} (cf. \cite{Ca97, BHHW00} for applications of this technique in mixed methods for the scalar Laplacian and electromagnetism).  Given $K \in \T_h$, let $b_K$ be the bubble function of polynomial degree $n+1$ obtained by multiplying the barycentric coordinates of $K$ together and scaling so that $\max_{x \in K} b_K=1$.  Extending by 0 outside of $K$ yields $b_K \in W_\infty^1(\Omega)$ with ${\rm supp}(b_K)=K$.  Similarly, given an $n-1$-dimensional face $e=K_1 \cap K_2$, where $K_1, K_2 \in \T_h$ and $K_2$ is void if $e \subset \partial \Omega$, we obtain an edge bubble function $b_e$ defined on $K_1, K_2$ by multiplying together the corresponding barycentric coordinates except that corresponding to $e$ and scaling so that $\max_{K_i} b_e=1$. 

Given a polynomial form $v$ of arbitrary but uniformly bounded degree defined on either $K \in \T_h$ or a face $e \subset K \in \T_h$, 
\begin{align} \label{eq5-1}
\|v \|_{K} \simeq \|\sqrt{b_K} v\|_{K},~~\|v\|_{e} \simeq \|\sqrt{b_e} v\|_{e}.
\end{align}
Also, given a polynomial $k$-form $v$ defined on a face $e=K_1 \cup K_2$, we wish to define a polynomial extension $\chi_v$ of $v$ to $K_1 \cup K_2$.   First extend $v$ in the natural fashion to the plane containing $e$.   We then extend $v$ to $K_i$, $i=1,2$ by taking $\chi_v$ to be constant in the direction normal to $e$.   Shape regularity implies that $e$, $K_1$, and $K_2$ are all contained in a ball having diameter equivalent to $h_K=h_{K_1} \simeq h_{K_2}$, so that an elementary computation involving inverse inequalities yields
\begin{align}\label{eq5-2}
\|\chi_v\|_{L_2(K_1 \cup K_2)} \lesssim h_K^{1/2} \|v\|_{L_2(e)}.
\end{align}

\subsection{Efficiency of $\eta_{-1}$}  We first consider the error indicator $\eta_{-1}$.  

\begin{lemma} \label{lem5-1}
Let $K \in \T_h$.  Then for $1 \le k \le n$,
\begin{align} \label{eq5-3}
\eta_{-1}(K) \lesssim \|e_\sigma\|_{\omega_K}+ \|e_u\|_{\omega_K}.
\end{align}
\end{lemma}

\begin{proof}
We begin with the term $h_K \|\sigma_h-\delta u_h\|_K$.  Let $\psi = b_K (\sigma_h -\delta u_h) \in H \Lambda^{k-1}$; note that $\trace_{\partial K} \psi=0$.  Employing \eqref{eq5-1}, the first line of \eqref{eq2-2}, and the integration-by-parts formula \eqref{eq3-3}, we obtain
\begin{align} \label{eq5-4}
\begin{aligned}
\|\sigma_h-\delta u_h\|_K^2 & \simeq \la \sigma_h-\delta u_h, \psi \ra=\la \sigma_h -\sigma, \psi \ra + \la d \psi, u-u_h \ra 
\\ &  \lesssim \|e_\sigma \|_K \|\psi\|_K + \|e_u\|_{K} \|d \psi\|_K.
\end{aligned}
\end{align}
Employing an inverse inequality and $b_K \le 1$ yields $\|d \psi \|_K \lesssim h_K^{-1} \|\psi \|_K \lesssim h_K^{-1} \|\sigma_h -\delta u_h \|_K$.  Multiplying \eqref{eq5-4} through by $h_K/\|\sigma_h-\delta u_h\|_K$ while noting that $h_K \lesssim 1$ yields
\begin{align} \label{eq5-5}
h_K \|\sigma_h - \delta u_h \|_K \lesssim \|e_\sigma \|_K + \|e_u\|_K.
\end{align}
Let now $k \ge 2$.  Recall that $\delta \sigma=\delta \delta u=0$.  Thus with $\psi = b_K \delta \sigma_h$, we have
\begin{align}\label{eq5-6}
\begin{aligned}
\|\delta \sigma_h \|_K^2 &  \simeq \la \delta \sigma_h, \psi \ra= \la \delta (\sigma_h -\sigma), \psi \ra=\la \sigma_h -\sigma, d \psi \ra 
\\ & \lesssim h_K^{-1} \|e_\sigma \|_K \|\psi \|_K \lesssim h_K^{-1} \| e_\sigma\|_K \|\delta \sigma_h \|_K.  
\end{aligned}
\end{align}
Multiplying through by $h_K/\|\delta \sigma_h \|_K$ yields
\begin{align}\label{eq5-7}
h_K \|\delta \sigma_h \|_K \lesssim \|e_\sigma \|_K.
\end{align}

We now consider edge terms.  Note from $\eqref{eq3-5}$ and the fact that $u \in H^*\Lambda^{k}(\Omega)$ (since $\delta u= \sigma \in L_2 \Lambda^{k-1}(\Omega)$) that we always have $\llbracket \trace \star u \rrbracket =0$ (suitably interpreted in $H^{-1/2}$).  Let $\mathcal{E}_h \ni e =K_1 \cap K_2$, where $K_2=\emptyset$ if $e \subset \partial \Omega$.  $\llbracket \trace \star u_h \rrbracket \in \Lambda^k(e)$, so we let $\psi \in \Lambda^{n-1-k}(e)$ satisfy $\star \psi = \llbracket \trace \star u_h \rrbracket$.  The definition of $\star$ implies that $\psi$ is a polynomial form because $ \llbracket \trace \star u_h \rrbracket$ is.  Note also that multiplication by $b_e$ commutes with $\trace$ and $\star$ since both are linear operations, so that $b_e \star \psi = \star (b_e \psi)=\star \trace (b_e \chi_\psi)$.  Employing the polynomial extension $\chi_\psi$ defined in \eqref{eq5-2} and surrounding along with the second relationship in \eqref{eq3-1-a} thus yields
\begin{align}\label{eq5-8}
\begin{aligned}
\|\llbracket \trace \star u_h \rrbracket \|_e^2  & \simeq \la b_e \star \psi, \llbracket \trace \star u_h \rrbracket \ra_e
\\ &  =\la  \star \trace (b_e \chi_\psi), \llbracket \trace \star u_h \rrbracket \ra_e = \int_e  \trace (b_e \chi_\psi) \wedge \llbracket \trace \star u_h \rrbracket.  
\end{aligned}
\end{align}  
Employing the integration-by-parts formula \eqref{eq3-3} individually on $K_1$ and $K_2$ yields
\begin{align}\label{eq5-9}
 \int_e  \trace (b_e \chi_\psi) \wedge \llbracket \trace \star u_h \rrbracket= \la d(b_e \chi_\psi), u_h \ra_{K_1 \cup K_2} - \la b_e \chi_\psi , \delta_h u_h \ra_{K_1 \cup K_2}.
\end{align}
Here $\delta_h$ is $\delta$ computed elementwise, which is necessary because $u_h \notin H^*\Lambda^k$ globally.  Also, $b_e \chi_\psi \in H \Lambda^{k-1}(\Omega)$ with support in $K_1 \cup K_2$.  Inserting the relationship $\la \sigma, b_e \chi_\psi \ra - \la d (b_e \chi_\psi), u \ra =0$ into \eqref{eq5-9} and using an inverse inequality $\|d (b_e \chi_\psi)\|_K \lesssim h_K^{-1} \|b_e \chi_\psi\|_K$, \eqref{eq5-2}, and the Hodge star isometry relationship $\|\psi\|_e \simeq \|\llbracket \trace \star u_h \rrbracket \|_e$ then yields
\begin{align}\label{eq5-10}
\begin{aligned}
\| \llbracket   & \trace \star   u_h  \rrbracket   \|_e^2  \simeq \la \sigma, b_e \chi_\psi \ra  - \la b_e \chi_\psi , \delta_h u_h \ra+\la d (b_e \chi_\psi), u_h-u \ra
\\ & = \la e_\sigma, b_e \chi_\psi \ra -\la b_e \chi_\psi, \sigma_h -\delta_h u_h \ra + \la d(b_e \chi_\psi), e_u \ra
 \\ & \lesssim \|b_e \chi_\psi \|_{K_1 \cup K_2} ( \|e_\sigma\|_{K_1 \cup K_2}  +h_K^{-1}\|e_u \|_{K_1 \cup K_2} + \|\sigma_h -\delta_h u_h \|_{K_1 \cup K_2})   
 \\ & \lesssim h_K^{1/2} \|\llbracket \trace \star u_h \rrbracket \|_{e} ( \|e_\sigma\|_{K_1 \cup K_2} + h_K^{-1}\|e_u\|_{K_1 \cup K_2} +\|\sigma_h -\delta u_h \|_{K_1 \cup K_2}).  
\end{aligned}
\end{align}
Multiplying both \eqref{eq5-10} through by $h_K^{1/2} / \|\llbracket \trace \star u_h \rrbracket \|_e$ and employing \eqref{eq5-5} thus finally yields
\begin{align} \label{eq5-12}
h_K^{1/2} \| \llbracket  \trace \star   u_h  \rrbracket   \|_e \lesssim \|e_u\|_{K_1 \cup K_2} + \|e_\sigma\|_{K_1 \cup K_2}.
\end{align}

A similar computation yields
\begin{align}\label{eq5-14}
h_K^{1/2} \|\llbracket \trace \star \sigma_h \rrbracket \|_e \lesssim \|e_\sigma\|_{K_1 \cup K_2},
\end{align}
thus completing the proof of Lemma \ref{lem5-1}.
\end{proof}

\subsection{Efficiency of $\eta_0$}

We next consider the various error indicators  $\eta_0$ in Lemma \ref{lem4-2}.  
First we define three types of data oscillation.  First, 
\begin{align}\label{eq5-15}
\osc (K)=h_K \|f-P f\|_K, 
\end{align}
where $P f$ is the $L_2$ projection of $f$ onto a space of polynomial $k$-forms of fixed but arbitrary degree.  Note that $P f$ is in general globally discontinuous.   
We do not specify the space further, since it is only necessary that it be finite dimensional in order to allow the use of inverse inequalities.  Similarly, we define the edge oscillation
\begin{align}\label{eq5-15-a}
\osc (\partial K)=h_K ^{1/2} \|\llbracket \trace \star (f-P f)\rrbracket \|_{L_2(\partial K)}.
\end{align}
Finally, we define
\begin{align}\label{eq5-15-b}
\osc_\delta (K)=h_K \|\delta (f-P f)\|_{L_2(K)}.
\end{align}
For a mesh subdomain $\omega$ of $\Omega$, let $\osc (\omega)= \left ( \sum_{T \subset \omega} \osc (K)^2 \right ) ^{1/2}$ and similarly for $\osc _\delta$.  The last two oscillation notions measure oscillation of $d \sigma$ only, since the Hodge decomposition yields $\llbracket \trace \star f \rrbracket = \llbracket \trace \star d \sigma \rrbracket$ and $\delta f = \delta d \sigma$.

\begin{lemma}  Let $K \in \T_h$, and consider the error indicators defined in Lemma \ref{lem4-2}.  For $k=0$, we have
\begin{align}\label{eq5-16}
\eta_0(K) \lesssim \|e_u\|_{H, \omega_K} + \osc(\omega_K)
\end{align}
When $k=n$, 
\begin{align}\label{eq5-17}
\eta_0(K) \lesssim \|d e_\sigma\|_K+ \|e_p\|_K.
\end{align}
For $1 \le k \le n-1$,
\begin{align} \label{eq5-18}
\begin{aligned}
\eta_0(K) \lesssim & \|e_u\|_{H, \omega_K} + \|e_\sigma\|_{H, \omega_K} + \|e_p\|_{\omega_K} 
\\ &~~~~+ \osc(\omega_K) + \osc_\delta (\omega_K) + \osc(\partial K).  
\end{aligned}
\end{align}

\end{lemma}

\begin{proof}
For the case $k=n$, \eqref{eq5-17} follows trivially from the Hodge decomposition $f=d \sigma + p$ and the triangle inequality.

For the case $0 \le k \le n-1$, let $\psi = b_K (P f -d \sigma_h -p_h -\delta d u_h)$.  Then
\begin{align} \label{eq5-19}
\begin{aligned}
\|P f -d \sigma_h -p_h -\delta d u_h\|_K^2 & \simeq  \langle P f -d \sigma_h -p_h -\delta d u_h, \psi \rangle_K
  \\ & = \langle Pf-f, \psi \rangle_K + \langle f-d \sigma_h -p_h -\delta d u_h, \psi \rangle_K .
  \end{aligned}
  \end{align}
Employing the Hodge decomposition $f=d\sigma + p + \delta d u$ and then integrating by parts while recalling that $b_K$ and thus $\psi$ vanishes on $\partial K$ yields
\begin{align}\label{eq5-20}
\begin{aligned}
\langle f-d\sigma_h-p_h-\delta d u_h, \psi \rangle _K & = \langle de_\sigma + e_p + \delta d e_u, \psi \rangle_K 
\\ & = \langle e_\sigma, \delta \psi \rangle_K + \langle e_p, \psi \rangle_K + \langle d e_u, d \psi \rangle _K.
\end{aligned}
\end{align}
Collecting \eqref{eq5-19} and \eqref{eq5-20} and then employing the inverse inequality $\|d \psi \|_K + \|\delta \psi\|_K \lesssim h_K^{-1} \|\psi\|_K$, multiplying the result through by $h_K$, and dividing through by $\|Pf-d\sigma_h -p_h-\delta d u_h\|_K$ after recalling that $\|\psi\|_K \lesssim \|Pf-d\sigma_h -p_h-\delta d u_h\|_K$ yields
\begin{align}\label{eq5-21}
h_K \|Pf-d\sigma_h -p_h-\delta d u_h\|_K \lesssim \|e_\sigma\|_K + \|d e_u\|_K + h_K \| e_p\|_K + {\rm osc}(K).
\end{align}
Employing the triangle inequality completes the proof that $h_K \|f-d\sigma_h-p_h -\delta du_h\|_K$ is bounded by the right hand side of \eqref{eq5-18} when $1 \le k \le n-1$, or by the right hand side of \eqref{eq5-16} when $k=0$ after noting that in this case $p=p_h$ is a constant and recalling that $\sigma-\sigma_h$ is vacuous.  

We next consider the term $h_K^{1/2} \|\llbracket \trace \star d u_h \rrbracket \|_{\partial K}$ in \eqref{eq4-20}.  
Manipulations similar to those in the previous subsection yield
\begin{align}
\label{eq5-23}
h_K^{1/2} \|\llbracket \trace \star d u_h \rrbracket\|_{e} \lesssim \|d e_u\|_{K_1 \cup K_2} + h_K \|\delta d e_u\|_{K_1 \cup K_2}.
\end{align}
Employing the Hodge decomposition $f=\d \sigma + p + \delta d u$ yields $\delta d e_u = (f -d \sigma_h -p_h -\delta d u_h)-d e_\sigma - e_p$.  Thus
\begin{align}\label{eq5-24}
\begin{aligned}
h_K^{1/2} & \|\llbracket \trace \star d u_h \rrbracket\|_{e}   \lesssim \|d e_u\|_{K_1 \cup K_2}
\\ & +h_K (\|f-d \sigma_h -p_h -\delta d u_h\|_{K_1 \cup K_2} + \|d e_\sigma\|_{K_1 \cup K_2} + \|e_p \|_{K_1 \cup K_2}).  
\end{aligned}
\end{align}
Employing \eqref{eq5-21} on $K_1$ and $K_2$ individually completes the proof that $h_K^{1/2} \|\llbracket \trace \star d u_h \rrbracket \|_e$ is bounded by the right hand side of \eqref{eq5-18} in the case $1 \le k \le n-1$ and of \eqref{eq5-16} when $k=0$.

We now consider the term $h_K \|\delta (f-d \sigma_h -p_h)\|_K$.  First note that $h_K \|\delta (f-d\sigma_h-p_h)\|_K \le {\rm osc}_{\delta}(K) + h_K \|\delta(Pf-d \sigma_h-p_h)\|_K$.  Letting $\psi=b_K \delta(P f-d \sigma_h-p_h)$ and recalling the identities $\delta f=\delta d \sigma$ and $\delta p=0$, we integrate by parts to compute
\begin{align}\label{eq5-25}
\begin{aligned}
\|\delta(Pf-d \sigma_h-p_h)\|_K^2 & \simeq \langle \delta (Pf-d \sigma_h-p_h),\psi\rangle
\\  &= \langle \delta (Pf-f), \psi \rangle + \langle \delta(d e_\sigma+e_p), \psi \rangle
\\ & \le h_K^{-1} \osc_\delta (K) \|\psi\|_K + |\langle d e_\sigma + e_p, d \psi \rangle|
\\ & \lesssim h_K^{-1} (\osc_\delta(K) + \|d e_\sigma\|_K + \|e_p\|_K) \|\psi\|_K.
\end{aligned}
\end{align}
Further elementary manipulations as in \eqref{eq5-19} and following complete the proof that $h_K \|\delta (f-d \sigma_h -p_h)\|_K$ is bounded by the right hand side of \eqref{eq5-18}. 

We finally turn to the edge term $h_K^{1/2} \|\llbracket \trace \star (f-d \sigma_h-p_h) \rrbracket \|_{\partial K}$.  Note first that $\llbracket \trace \star (p+\delta d u) \rrbracket =0$ on all element faces $e$.  On interior faces this is a result of the fact that $p+\delta d u \in H^* \Lambda^k$, while for boundary edges this is a result of \eqref{eq3-7} along with the definition of $\H^k$.  Thus $\llbracket \trace \star f \rrbracket = \llbracket \trace \star d \sigma \rrbracket$.  Setting $ \star \psi=\llbracket \trace \star (P f-d \sigma_h -p_h)\rrbracket$ and letting $\chi_\psi$ be the polynomial extension of $\psi$ as above, we compute for a face $e=K_1 \cap K_2$ that
\begin{align}\label{eq5-26}
\begin{aligned}
\| \llbracket & \trace  \star (Pf-d \sigma_h -p_h)\rrbracket \|_{e}^2   \simeq \langle b_e \star \psi, \llbracket \trace \star (Pf-d \sigma_h -p_h)\rrbracket  \rangle
\\ & \le h_K^{-1/2} \osc(\partial K) \|\psi\|_e + |\langle b_e \psi, \llbracket \trace \star (d e_\sigma -p_h \rrbracket  \rangle| 
\\ &  = h_K^{-1/2} \osc(\partial K) \|\psi\|_e + | \langle d(b_e \chi_\psi), d e_\sigma-p_h \rangle_{K_1 \cup K_2} 
\\ & ~~~ + \langle b_e \chi_\psi, \delta (d e_\sigma-p_h)\rangle_{K_1 \cup K_2}|.
\end{aligned}
\end{align} 
Next note that $\langle d(b_e \chi \psi), p \rangle =0$, so that $\langle d(b_e \chi \psi), -p_h \rangle = \langle d(b_e \chi \psi), e_p \rangle$.  Using an inverse inequality and \eqref{eq5-2} then yields 
\begin{align}\label{eq5-27}
\begin{aligned}
\| \llbracket  \trace  \star (Pf-d \sigma_h -p_h)\rrbracket \|_{e}^2 & \lesssim h_K^{-1/2} \Big [{\rm osc} ( \partial K) +  \|e_p\|_{K_1 \cup K_2} + \|d e_\sigma\|_{K_1 \cup K_2}
\\ & ~~~~+ h_K \|\delta (f-d\sigma_h -p_h)\|_{K_1 \cup K_2} \Big ] \|\psi \|_{e}.  
\end{aligned}
\end{align}
Further elementary manipulations as above complete the proof that $h_K^{1/2} \| \llbracket \trace  \star (f-d \sigma_h -p_h)\rrbracket \|_{e}$ is bounded by the right hand side of \eqref{eq5-18}.  

\end{proof}

\subsection{Efficiency of harmonic indicators}  We finally state efficiency results for the various harmonic terms.

In this section we prove efficiency of the individual harmonic terms appearing in Lemma \ref{lem4-3}.  As we discuss more thoroughly below, however, we do not obtain efficiency of all of the terms that we originally sought to bound. 

\begin{lemma}\label{lem5-3}
Let $v_h \in V_h^k$.  Then
\begin{align}
\label{eq5-28}
\eta_{\H}(K, v_h) \lesssim \|P_\B v_h\|_{\omega_K}.  
\end{align}
In particular, we have for $u_h^\perp$, $q_i \in \H_h^k$, and $p_h$
\begin{eqnarray}
\label{eq5-29}
\eta_\H (K, u_h^\perp) &\lesssim& \|P_\B u_h^\perp\|_{\omega_K},
\\
\label{eq5-30}
\eta_\H(K, q_i) &\lesssim& \|P_\B q_i\|_{\omega_K}=\|q_i-P_\H q_i\|_{\omega_K},
\\
\label{eq5-30a}
\eta_\H(K, p_h) &\lesssim& \|e_p\|_{\omega_K}.
\end{eqnarray}
Thus with $\mu$ and $\mu_i$ as in Lemma \ref{lem4-3},
\begin{align}
\label{eq5-31}
\mu \lesssim {\rm gap}(\H^k, \H_h^k)
\end{align}
\end{lemma}

\begin{proof}
The proof of \eqref{eq5-28} is a straightforward application of the bubble function techniques used in the previous subsections.  \eqref{eq5-29} and \eqref{eq5-30} are special cases of \eqref{eq5-28}, while \eqref{eq5-30a} may be proved similarly.  Finally,  summing \eqref{eq5-30} in $\ell_2$ over $K \in \T_h$ while employing the finite overlap of the patches $\omega_K$ (which is a standard consequence of shape regularity) implies that
\begin{align}\label{eq5-32}
\mu_i \lesssim \|q_i-P_\H q_i\|_{\omega_K},
\end{align}
which yields \eqref{eq5-31} when summed over $1 \le i \le M$.  
\end{proof}

\begin{remark}{\rm Lemma \ref{lem5-3} gives efficiency results for the terms in our a posteriori bounds for ${\rm gap}(\H^k, \H_h^k)$ and for $\|P_\B u_h^\perp\|$, but not for the quantity $\|P_\H u_h\|$ itself that we originally sought to bound.  More generally, we have not bounded all of the harmonic terms \eqref{eq4-60} and \eqref{eq4-61} by the error on the left hand side of \eqref{eq4-60} as would be ideal.  The offending terms 
are due to the nonconforming nature of our method which arises from the fact that $\H_h^k \not= \H^k$.  Establishment of efficiency of reliable estimators for this harmonic nonconformity error remains an open problem.  
}\end{remark}


\section{Examples}\label{sec_examples}

In this section we translate our results into standard notation for a posteriori error estimators in the context of the canonical three-dimensional Hodge-de Rham Laplace operators.  Below we always assume that $n=3$.

\subsection{The Neumann Laplacian}  When $k=0$, $\sigma$ and the first equation in \eqref{eq2-2}) are vacuous.  Also, $V^{k-1}=V^{-1}=\emptyset$, $V^k=V^0=H^1(\Omega)$, $d=\nabla$, and $\delta = -\Div$.  In addition, $p=p_h=\dashint_\Omega f$, and $\delta d u=-\Delta u$.  
The weak mixed problem \eqref{eq2-2} reduces to the standard weak form of the Laplacian and naturally enforces homogeneous Neumann boundary conditions.   In Lemma \ref{lem4-1} and Lemma \ref{lem4-3} we have $\eta_{-1}\equiv 0$ and  $\mu=\|P_\B u_h^\perp\|=0$, respectively.  
Thus $\eta_1$ is thus the only nontrivial indicator for this problem,  and it reduces to the standard indicator $\eta(K)$ from \eqref{eq1-3}.  Thus we recover standard results for the Neumann Laplacian.

\subsection{The vector Laplacian:  $k=1$}  For $k=1$ and $k=2$, the Hodge Laplacian corresponds to the vector Laplace operator $\curl \curl - \nabla \Div$, but with different boundary conditions.  For $k=1$, $u \in H(\curl)$, $\sigma = -\Div u \in H^1$, and the boundary conditions are $u\cdot n=0$, $\curl u \times n=0$ on $\partial \Omega$.   $\H^1$ consists of vector functions $p$ satisfying $\curl p=0$, $\Div p=0$ in $\Omega$ and $p \cdot n =0$ on $\partial \Omega$.  From \eqref{eq4-1}, 
\begin{align}\label{eq6-0}
\eta_{-1}(K)=h_K \|\sigma_h+\Div u_h\|_K+ h_K^{1/2} \|\llbracket u_h \cdot n\rrbracket \|_{\partial K}.
\end{align}
Here $n$ is a unit normal on $\partial K$.  From \eqref{eq4-20} we find
\begin{align}\label{eq6-0-a}
\begin{aligned}
\eta_0(K)& =h_K (\|f-\nabla \sigma_h-p_h -\curl \curl u_h\|_K + \|\Div (f-\nabla \sigma_h-p_h)\|_K)
\\ & ~~~+ h_K^{1/2} ( \|\llbracket (\curl u_h)_t \rrbracket \|_{\partial K} + \| \llbracket (f-\nabla \sigma_h-p_h) \cdot n \rrbracket \|_{\partial K}),
\end{aligned}
\end{align}
where a subscript $t$ denotes a tangential component.  
Finally, in Lemma \ref{lem4-3} we have
\begin{align} \label{eq6-0-b}
\eta_\H (K, q_h)=h_K \|\Div q_h\|_K + h_K^{1/2} \|\llbracket q_h \cdot n \rrbracket \|_{\partial K}.
\end{align}

\subsection{The vector Laplacian:  $k=2$}  In the case $k=2$ the mixed form of the vector Laplacian yields $\sigma = \curl u$, $u \in H(\Div)$, and $u \times n=0$, $\Div u=0$ on $\partial \Omega$.  In addition, $\H^2$ consists of vector functions $p$ satisfying $\curl p=0$, $\Div p=0$ in $\Omega$ and $p \times n=p_t=0$ on $\partial \Omega$.  We then have from \eqref{eq4-1} that
\begin{align} \label{eq6-0-c}
\begin{aligned}
\eta_{-1}(K) & = h_K (\|\Div \sigma_h \|_K + \|\sigma_h -\curl u_h \|_K) 
\\ &~~~+ h_K^{1/2} (\|\llbracket \sigma_{h} \cdot n \rrbracket \|_{\partial K} + \|\llbracket u_{h,t} \rrbracket \|_{\partial K}).
\end{aligned}
\end{align}
From \eqref{eq4-20} we have
\begin{align}\label{eq6-0-d}
\begin{aligned}
\eta_0(K)& = h_K ( \|f - \curl \sigma_h -p_h + \nabla \Div u_h \|_K + \|\curl ( f- \curl \sigma_h - p_h) \|_K) 
\\ & ~~~ h_K^{1/2} (\| \llbracket \Div u_h \rrbracket \|_{\partial K} + \|\llbracket (f-\curl \sigma_h - p_h)_t\rrbracket \|_{\partial K}). 
\end{aligned}
\end{align}
Finally, in Lemma \ref{lem4-3} we have
\begin{align}\label{eq6-0-e}
\eta_\H (K, q_h)= h_K \|\curl q_h \|_K + h_K^{1/2} \|\llbracket q_{h,t} \rrbracket \|_{\partial K}.
\end{align}

\subsection{Mixed form of the Dirichlet Laplacian} \label{subsec_classicalmixed}  For $k=3$, \eqref{eq2-2} is a standard mixed method for the Dirichlet Laplacian $-\Delta u=0$ in $\Omega$, $u=0$ on $\partial \Omega$, and $\sigma=-\nabla u$.  $d^2=\Div$, $d^3$ is vacuous, $\H^3=\H_h^3=\emptyset$, $V^{k-1}=H(\Div)$, and $V^k=L_2$.   Taking $\sigma_h$ and $u_h$ now to be proxy vector fields for $\sigma_h$ and $u_h$, we have in \eqref{eq4-1} that $\delta \sigma_h=\curl \sigma_h$, $\delta u_h=-\nabla u_h$, $\trace \star \sigma_h=\sigma_{h,t}$ (i.e., the tangential component of $\sigma_h$), and $\trace \star u_h = u_h$.
\begin{align}\label{eq6-1}
\eta_{-1}(K)=h_K (\|\curl \sigma_h\|_K + \|\sigma_h+\nabla u_h\|_{K}) + h_K^{1/2}( \|\llbracket \sigma_{h,t}\rrbracket \|_K+ \|\llbracket u_h\rrbracket \|_K).
\end{align}
In addition, \eqref{eq4-20} yields
\begin{align}\label{eq6-2}
\eta_0(K)=\|f-\Div \sigma_h\|_K.
\end{align}
The ``harmonic estimators'' in Lemma \ref{lem4-3} are all vacuous in this case.  Combining Theorem \ref{t4-4} with the corresponding efficiency bounds of \S \ref{sec_efficiency} thus yields
\begin{align}\label{eq6-3}
\|e_u\|_{L_2(\Omega)} + \|e_\sigma\|_{H(\Div; \Omega)} \simeq (\sum_{K \in \T_h} \eta_{-1}(K)^2 + \eta_0(K)^2)^{1/2}.
\end{align}

In contrast to the vector Laplacian, many authors have proved a posteriori error estimates for the mixed form of Poisson's problem, so we compare our results with existing ones.  We focus mainly on two early works bounding a posteriori the natural mixed variational norm $H(\Div) \times L_2$.  In \cite{BrVe96}, Braess and Verf\"urth prove a posteriori estimates for $\|e_\sigma\|_{H(\Div; \Omega)} + \|e_u\|$, as we do here, but their estimates are only valid under a saturation assumption (which is not a posteriori verifiable) and are not efficient.  Salient to our discussion is their observation on pp. 2440--2441 that the traces of $H(\Div)$ test functions lie only in $H^{-1/2}$.  This prevented them from employing the mixed variational form in a straightforward way, that is, using an inf-sup condition in order to test with functions in $H(\Div) \times L_2$.  Doing so using their techniques would have led to a duality relationship between traces lying in incompatible spaces, or more particularly, between traces lying in $H^{-1/2}$ and some space less regular than $H^{1/2}$.  Following ideas used in \cite{Ca97} in the context of the mixed scalar Laplacian and developed more fully in \cite{Sch08} for Maxwell's equations, we insert the essential additional step of first taking the Hodge decomposition of test functions.  Only the regular ($H^1$) portion of the test function is then integrated by parts, thus avoiding trace regularity issues.  Note finally that the elementwise indicators of \cite{BrVe96} are of the form $\|\Div \sigma_h-f\|_K + \|\sigma_h + \nabla u_h \|_K + h_K^{-1/2} \|\llbracket u_h \rrbracket \|_{\partial K}$, which includes our indicator $\eta_0$ and parts of our indicator $\eta_{-1}$.  However, the jump term $h_K^{-1/2} \|\llbracket u_h \rrbracket \|_{\partial K}$ is scaled too strongly (by $h_K^{-1/2}$ instead of $h_K^{1/2}$ in our estimator), and the resulting bounds are thus not efficient; cf. (4.20) of \cite{BrVe96}.  

In \cite{Ca97} Carstensen provided a posteriori estimators for the natural $H(\Div) \times L_2$ norm which are equivalent to the actual error as in \eqref{eq6-3}.   In our notation, Carstensen's elementwise indicators have the form $\|f-\Div \sigma_h\|_K+ h_K \|\curl \sigma_h \|_K + h_K \min_{v_h \in L_h}\|\sigma_h + \nabla v_h\|_K + h_K^{1/2} \|\llbracket \sigma_{h,t} \rrbracket \|_{\partial K}$.  Here $L_h$ is an appropriate space of piecewise polynomials.  Thus our terms $h_K \| \sigma_h + \nabla u_h\|_K+ h_K^{1/2} \|\llbracket  u_h \rrbracket \|_{\partial K}$ are replaced in Carstensen's work by $h_K \min_{v_h \in L_h}\|\sigma_h + \nabla v_h\|_K$, and our estimators are otherwise the same.   However, Carstensen's results were proved only under the restrictive assumption that $\Omega$ is convex, which we avoid.  \cite{Ca97} also makes use of a Helmholtz (Hodge)  decomposition, but a {\it commuting} quasi-interpolant was not available at the time and thus full usage of the Hodge decomposition was not possible.  
 
Most works on a posteriori error estimation for mixed methods subsequent to \cite{Ca97} have focused on measuring the error in other norms, e.g., $\|e_\sigma\|_{L_2}$ (cf. \cite{LaMa08, Vo07}).  One essential reason for this is that the $H(\Div) \times L_2$ norm includes the term $\|f-\Div \sigma_h\|$ which directly approximates the data $f$ and which can thus be trivially computed a posteriori.  Thus while the $H(\Div) \times L_2$ norm is natural to consider from the standpoint of the mixed variational formulation, it is perhaps not the most important error measure in practical settings.  Even with this caveat, our estimators for mixed methods  for the Dirichlet Laplacian seem to be the first estimators that are directly proved to be reliable and efficient for the natural mixed variational norm under reasonably broad assumptions on the domain geometry.

\section*{Acknowledgements}

We would like to thank Doug Arnold for helpful discussions concerning Lemma \ref{lem_gap} and Gantumur Tsogtgerel for helpful discussions concerning Lemma \ref{lem3-1}.

\bibliographystyle{siam}
\bibliography{hirani,demlow}


\end{document}